\documentclass[11pt]{article}

\usepackage{amssymb}
\usepackage{mathrsfs}
\usepackage{latexsym,amsmath,amssymb,amsfonts,epsfig,graphicx,cite,psfrag}
\usepackage[shortlabels]{enumitem}
\usepackage{eepic,color,colordvi,amscd}
\usepackage{ebezier}
\usepackage{verbatim}
\usepackage{subfigure}
\usepackage{accents}
\usepackage{amsthm}
\usepackage{comment}
\usepackage{tikz}
\usepackage{todonotes}
\usepackage{lipsum}

\usepackage[colorlinks=true,
linkcolor=blue,citecolor=blue,
urlcolor=blue]{hyperref}
\usetikzlibrary{calc}

\theoremstyle{plain}
\newtheorem{theo}{Theorem}[section]
\newtheorem{claim}{Claim}

\newtheorem{lem}[theo]{Lemma}

\newtheorem{conj}[theo]{Conjecture}

\newtheorem{obs}{Observation}

\theoremstyle{definition}

\newcommand\blfootnote[1]{%
  \begingroup
  \renewcommand\thefootnote{}\footnote{#1}%
  \addtocounter{footnote}{-1}%
  \endgroup
}

\theoremstyle{remark}

\title{Counting Hamiltonian cycles in planar triangulations}
\author{Xiaonan Liu, Zhiyu Wang, and Xingxing Yu \footnote{Partially supported by NSF Grant DMS 1954134.}
\\School of Mathematics\\Georgia Institute of Technology\\Atlanta, GA 30332 USA}
\date{}

\begin{document}
\maketitle
\begin{abstract}
Hakimi, Schmeichel, and Thomassen [\textit{J. Graph Theory, 1979}] conjectured that every $4$-connected planar triangulation $G$ on $n$ vertices has at least $2(n-2)(n-4)$ Hamiltonian cycles, with equality if and only if $G$ is a double wheel. In this paper, 
we show that every $4$-connected planar triangulation on $n$ vertices has $\Omega(n^2)$ Hamiltonian cycles. Moreover, we show that if $G$ is a $4$-connected planar triangulation on $n$ vertices and the distance between any two vertices of degree $4$ in $G$ is at least $3$, then $G$ has $2^{\Omega(n^{1/4})}$ Hamiltonian cycles. 
\end{abstract}
\blfootnote{Corresponding Author: Xiaonan Liu.}

\section{Introduction}
A graph $G$ is called \textit{Hamiltonian} if it contains a \textit{Hamiltonian cycle}, i.e., a cycle that contains every vertex in $G$. A $k$-vertex cycle (or a \textit{k-cycle}) $C$ in a connected graph $G$ is said to be \textit{separating} if the graph obtained from $G$ by deleting $C$ is not connected. A separating $3$-cycle is also called a \textit{separating triangle}. A graph $G$ is \textit{$k$-connected} if it has more than $k$ vertices and if it remains connected when fewer than $k$ vertices are removed.  The \textit{distance} between two vertices in a graph is the number of edges in a shortest path in the graph connecting them.
A \textit{planar triangulation} is an edge-maximal plane graph with at least three vertices, i.e., every face is bounded by a triangle. By Euler's theorem, an $n$-vertex planar triangulation has exactly $3n-6$ edges. 

Whitney \cite{Whitney1931} showed in 1931 that every planar triangulation without separating triangles is Hamiltonian. In 1956, Tutte \cite{Tutte1956} extended Whitney's result by showing that every $4$-connected planar graph is Hamiltonian. Thomassen \cite{Thomassen1983} further strengthened Tutte's result in 1983 by showing that every $4$-connected planar graph is \textit{Hamiltonian connected}, i.e., any two distinct vertices are connected by a Hamiltonian path. These results have been extended to graphs on other surfaces, see {\color{blue} e.g.,}\cite{Altshuler1972, Thomas-Yu1994, Thomas-Yu1997, TYZ2005}. It is possible that results there can be combined with the methods in this paper to obtain similar counting results on Hamiltonian cycles in certain triangulations on other surfaces.

It is a natural problem to consider the number of cycles in a graph, which also has applications in coding theory according to  \cite{AACST2013a, AACST2013b, AACOST2015}. The problem of determining the number of Hamiltonian cycles in $4$-connected planar triangulations was initated by Hakimi, Schmeichel, and Thomassen \cite{HST1979} who showed in 1979 that every $4$-connected planar triangulation on $n$ vertices has at least $n/\log_2 n$ Hamiltonian cycles. In the same paper, they conjectured a lower bound which is quadratic in the number of vertices and realized by the double wheel. A {\it double wheel} is a planar triangulation obtained from a cycle by adding two vertices and all edges from these two vertices to the vertices of the cycle.

\begin{conj}[Hakimi, Schmeichel, and Thomassen \cite{HST1979}]\label{conj:HST}
If $G$ is a $4$-connected planar triangulation on n vertices, then $G$ has at least $2(n -2)(n - 4)$ Hamiltonian cycles, with equality if and only if $G$ is a double wheel.
\end{conj}

It was not until recently that Brinkmann, Souffriau, and Van Cleemput \cite{BSC2018} gave the first linear lower bound of $\frac{12}{5}(n-2)$ for $n$-vertex $4$-connected planar triangulations. Subsequently, Brinkmann and Van Cleemput \cite{BC2021} proved  linear lower bounds for 4-connected plane graphs and plane graphs with at most one 3-cut and sufficiently many edges.
 Since then, there has been more progress on this problem. Lo \cite{Lo2020} showed that every $4$-connected $n$-vertex planar triangulation with $O(\log n)$ separating $4$-cycles has $\Omega((n/\log n)^2)$ Hamiltonian cycles. The first and third author \cite{Liu-Yu2021} further showed that every $n$-vertex $4$-connected planar triangulation with $O(n/\log_2 n)$ separating $4$-cycles has $\Omega(n^2)$ Hamiltonian cycles. Very recently, Lo and Qian \cite{LQ2021} showed that every $n$-vertex $4$-connected planar triangulation with $O(n)$ separating $4$-cycles has $2^{\Omega(n)}$ Hamiltonian cycles.

Thus Conjecture~\ref{conj:HST} holds for large graphs with $O(n)$ separating $4$-cycles. In this paper, we remove the assumption on separating $4$-cycles and settle Conjecture~\ref{conj:HST} asymptotically.

\begin{theo}\label{main1}
If $G$ is a $4$-connected planar triangulation on $n$ vertices, then $G$ has at least $c n^2$ Hamiltonian cycles, where $c= (12\times 90\times 541\times 301)^{-2}/2$.
\end{theo}

The number of Hamiltonian cycles in a planar triangulation $G$ can be significantly larger if one increases the connectivity or the minimum degree of $G$. Alahmadi, Aldred, and Thomassen \cite{AAT2020} showed that every $5$-connected $n$-vertex planar triangulation has $2^{\Omega (n)}$ Hamiltonian cycles, improving the earlier bound $2^{\Omega(n^{1/4})}$ by B\"{o}hme, Harant, and Tk{\' a}{\v c} \cite{BHT1999}. Note that the more recent result of Lo and Qian \cite{LQ2021} is stronger, but the technique used in \cite{AAT2020} played an important role in \cite{LQ2021}. The first and third author \cite{Liu-Yu2021} weakened the assumption in the  B\"{o}hme-Harant-Tk{\' a}{\v c} result by replacing the $5$-connectedness condition with ``minimum degree at least $5$''. In this paper, we observe that the relative locations of degree $4$ vertices play an essential role for $4$-connected planar triangulations to have exponentially many Hamiltonian cycles.

\begin{theo}\label{main2}
There exists a constant $c>0$ such that 
for any $4$-connected planar triangulation $G$ on $n$ vertices in which the distance between any two vertices of degree $4$ is at least three, $G$ has at least $2^{c n^{1/4}}$ Hamiltonian cycles.
\end{theo}

In Section $2$, we discuss an idea similar to the key idea in \cite{AAT2020} for finding an edge set $F$ in a $4$-connected planar triangulation $G$ such that removing $F$ from $G$ still gives a $4$-connected graph. We collect several results on the number of Hamiltonian paths between two given vertices in planar graphs. We also cite some known results on ``Tutte paths'' and ``Tutte cycles" in planar graphs. Such results will be used to find a Hamiltonian cycle through specific edges in a planar graph. 

In Section $3$, we prove Theorem~\ref{main1}. We first show that every  $n$-vertex  $4$-connected planar triangulation $G$ has $\Omega(n)$ Hamiltonian cycles through two specified edges in any given triangle. Moreover, if $G$ does not contain two adjacent vertices of degree $4$, then $G$ has $\Omega({n^2})$ Hamiltonian cycles.
We then use these results and apply induction on $n$ to complete the proof of Theorem \ref{main1}.

In Section $4$, we consider $4$-connected planar triangulations $G$ in which any two vertices of degree $4$ have distance at least three.  We show that either $G$ has a large independent set with nice properties, or $G$ has many separating $4$-cycles with pairwise disjoint interiors, or $G$ has many ``nested" separating $4$-cycles. In all cases, we can find the desired number of Hamiltonian cycles in $G$.  
\medskip 

We conclude this section with some terminology and notation. For any positive integer $k$, let $[k]=\{1,2, \ldots, k\}$.

Let $G$ and $H$ be graphs. We use $G\cup H$ and $G\cap H$ to denote the union and intersection of $G$ and $H$, respectively. For any $S\subseteq V(G)$, we use $G[S]$ to denote the subgraph of $G$ induced by $S$, and let $G-S = G[V(G)\backslash S]$. A set $S\subseteq V(G)$ is a {\it cut} in $G$ if $G-S$ has more components than $G$, and if $|S|=k$ then $S$ is a cut of {\it size $k$} or {\it $k$-cut} for short.  For a subgraph $T$ of $G$, we often write $G-T$ for $G-V(T)$ and write $G[T]$ for $G[V(T)]$. 
A path (respectively, cycle) is often represented as a sequence (respectively, cyclic sequence) of vertices, with consecutive vertices 
being adjacent.  Given a path $P$ and distinct vertices $x,y\in V(P)$, we use $xPy$ to denote the subpath of $P$ between $x$ and $y$.

Let $G$ be a graph. For $v\in V(G)$, we use $N_G(v)$ (respectively, $N_G[v]$) to denote the neighborhood (respectively, closed neighborhood) of $v$, and use $d_G(v)$ to denote $|N_G(v)|$. For distinct vertices $u,v$ of $G$, we use $d_G(u,v)$ to denote the distance between $u$ and $v$, and if $u$ and $v$ are adjacent in $G$, we use $uv$ to denote the edge of $G$ between $u$ and $v$. (If there is no confusion we omit the reference to $G$.)
If $H$ is a subgraph of $G$, we write  $H \subseteq G$. For any set $R$ consisting of vertices of $G$ and $2$-element subsets of $V(G)$, we use $H+R$ (respectively, $H-R$) to denote the graph with vertex set $V(H)\cup (R\cap V(G))$ (respectively, $V(H)\backslash (R\cap V(G))$) and edge set $E(H)\cup (R\backslash V(G))$ (respectively, $E(H)\backslash (R\backslash V(G))$). If $R=\{\{x,y\}\}$ (respectively, $R=\{v\}$), we write $H+xy$ (respectively, $H+v$)  instead of $H+R$, and write $H-xy$ (respectively, $H-v$) instead of $H-R$.

Let $G$ be a plane graph. Two elements of $V(G)\cup E(G)$ are {\it cofacial} if they are incident with a common face of $G$. The {\it outer walk} of $G$ consists of vertices and edges of $G$ incident with the infinite face of $G$. If the outer walk is a cycle in $G$, we call it  {\it outer cycle} instead.  If all vertices of $G$ are incident with its infinite face, then we say that $G$ is an {\it outer planar} graph. 
For a cycle $C$ in  $G$, we use $\overline{C}$ to denote the subgraph of $G$ consisting of all vertices and edges of $G$ contained in the closed disc in the plane bounded by $C$. The {\it interior} of $C$ is then defined as the subgraph $\overline{C}-C$. For any distinct vertices $u,v\in V(C)$, we use $uCv$ to denote the subpath of $C$ from $u$ to $v$ in clockwise order.

\section{Preliminaries}

In this section, we state and prove a number of lemmas needed for the proofs of Theorems~\ref{main1} and \ref{main2}.

A \textit{near triangulation} is a plane graph in which all faces except possibly its infinite face are bounded by triangles. The first and third author \cite{Liu-Yu2021} considered the number of Hamiltonian paths between two given vertices in the outer cycle of a near triangulation.

\begin{lem}[Liu and Yu \cite{Liu-Yu2021}]\label{uw-path}
Let $G$ be a near triangulation with outer cycle $C:=uvwxu$ and assume that $G\ne C+vx$ and  $G$ has no separating triangles. Then one of the following holds:
\begin{itemize}
\item[\textup{(i)}] $G-\{v, x\}$ has at least two Hamiltonian paths between $u$ and $w$.
\item[\textup{(ii)}] $G-\{v,x\}$ is a path between  $u$ and $w$ and, hence, $G-\{v,x\}$ is outer planar.
\end{itemize}
\end{lem}
\begin{lem}[Liu and Yu \cite{Liu-Yu2021}] \label{uv-path}
Let $G$ be a near triangulation with outer cycle $C: =uvwxu$ and assume that $G$ has no separating triangles. Then one of the following holds:
\begin{itemize}
    \item [\textup{(i)}]  $G-\{w,x\}$ has at least two Hamiltonian paths between $u$ and $v$.
    \item [\textup{(ii)}]  $G-\{w,x\}$ is an outer planar near triangulation. 
\end{itemize}
\end{lem}

We need an observation about degree $4$ vertices and the number of Hamiltonian paths in a near triangulation.
\begin{lem}\label{lem:2Hamiltonpaths}
Let $G$ be a near triangulation with outer cycle $C:=uvwxu$ and assume that $|V(G)|\geq 6$ and $G$ has no separating triangles. Suppose there exist distinct $a,b\in V(C)$ such that $G-(V(C)\backslash \{a,b\})$ has at most one Hamiltonian path between $a$ and $b$. Then $G$ has two adjacent vertices of degree $4$ that are contained in $V(G)\backslash V(C)$. 
\end{lem}
\begin{proof}
By symmetry, we only need to consider two cases: $\{a,b\}=\{u,w\}$ or $\{a,b\}=\{u,v\}$. If $\{a,b\}=\{u,w\}$ and $G-(V(C)\backslash \{a,b\})=G-\{v,x\}$ has at most one Hamiltonian path between $u$ and $w$, then by Lemma~\ref{uw-path}, $G-\{v,x\}$ is a path. Hence, by planarity, all vertices in $V(G)\backslash V(C)$ have degree $4$ in $G$; so the assertion holds as $|V(G)\backslash V(C)|\geq 2$. 

Now suppose $\{a,b\}=\{u,v\}$ and there exists at most one Hamiltonian path between $u$ and $v$ in $G-(V(C)\backslash \{a,b\})=G-\{w,x\}$. Then by Lemma~\ref{uv-path}, $G-\{w,x\}$ is an outer planar near triangulation. Let $D=u_1 u_2 \ldots u_tu_1$ denote the outer cycle of $G-\{w,x\}$ such that $u_1=u$ and $u_t=v$. Note that $t\geq 4$ and that $u_i$ is adjacent to $w$ or $x$ for every $i\in [t]$.  Let $u_{s}$, where $s\in [t]$, be the common neighbor of $w$ and $x$ in $V(D)$. (The existence of $u_s$ is guaranteed by the fact that $G$ is a near triangulation with outer cycle $uvwxu$.) Since $G$ has no separating triangles, $2\leq s \leq t-1$ and every edge of $(G-\{w,x\})-E(D)$ is incident with both paths $u_1\ldots u_{s-1}$ and $u_{s+1}\ldots u_t$. It follows that $d_G(u_{s})=4$. Moreover, $s\geq 3$ and $d_G(u_{s-1})=4$, or $s\leq t-2$ and $d_G(u_{s+1})=4$, as $|V(G)\backslash V(C)| \geq 2$ and $G$ is a near triangulation. This completes the proof of the lemma.
\end{proof}

By Lemma~\ref{lem:2Hamiltonpaths}, it is natural to expect that $4$-connected planar triangulations without too many vertices of degree $4$ should have many Hamiltonian cycles.
We now prove a technical lemma, which will be used in the proof of Lemma \ref{nest4cycles} to produce two Hamiltonian paths in a near triangulation.
\begin{lem}\label{lem:r-yHamiltonian}
Let $G$ be a near triangulation with outer cycle $C$, and let $x_1,w_1,w_2, x_2\in V(C)$ be distinct and occur on $C$ in clockwise order such that $x_1x_2, w_1w_2 \in E(C)$ and each edge of $G-E(C)$ is incident with both $x_1Cw_1$ and  $w_2Cx_2$. Let $N_G(x_1)\cap N_G(x_2)=\{r\}$ and $N_G(w_1)\cap N_G(w_2)=\{y\}$, and assume $r\notin\{y, w_1,w_2\}$ and $y\notin \{r,x_1,x_2\}$. Suppose any two degree $3$ vertices of $G$ contained in $V(G)\backslash \{x_1,x_2,w_1,w_2\}$ have distance at least three in $G$. Then $G-\{x_1,x_2,w_1,w_2\}$ has a Hamiltonian path between $r$ and $y$.
\end{lem}
\begin{proof}
Note that $|V(G)|\geq 6$ as $r\notin\{y, w_1,w_2\}$ and $y\notin\{r,x_1,x_2\}$. We apply induction on $|V(G)|$. Without loss of generality, we may assume $r\in V(x_1Cw_1)$. Then $d_G(x_1)=2$.

Suppose $|V(G)|=6$. If $ry\in E(G)$ then we are done. So assume $ry\notin E(G)$. Then $y\in V(w_2Cx_2)$, $x_2w_1\in E(G)$, and $d_G(r)=d_G(y)=3$. This gives a contradiction since $d_G(r,y)=2$.

Now assume $|V(G)|>6$. We have two cases: $y\in V(x_1Cw_1)$ or $y\in V(w_2Cx_2)$.

\medskip

{\it Case} 1. $y\in V(x_1Cw_1)$. Then $d_G(w_1)=2$.

Consider $G_1=G-w_1$. Let $y'$ denote the unique vertex in $N_{G_1}(y)\cap N_{G_1}(w_2)$. If $y'\notin \{r,x_2\}$ then, by induction, $G_1-\{x_1,x_2, y, w_2\}$ has a Hamiltonian path $H_1$ between $r$ and $y'$; so $H_1 + y'y$ gives a Hamiltonian path between $r$ and $y$ in $G- \{x_1, x_2, w_1, w_2\}$. Hence, we may assume $y'\in \{r,x_2\}$. 

If $y'=r$ then $ry,rw_2\in E(G)$ and $d_G(y)=3$. Since $|V(G)|\geq 7$, $|V(w_2Cx_2)|\geq 3$. Now, $y$ and a degree $3$ vertex of $G$ contained in $V(w_2Cx_2)\backslash \{x_2,w_2\}$ are distance $2$ apart in $G$, a contradiction. 
So $y'=x_2$. Then $x_2y,x_2w_2\in E(G)$. Hence, since each edge of $G-E(C)$ is incident with both $x_1 C w_1$ and $x_2 C w_2$, $V(rCy) \subseteq N_G(x_2)$, and all vertices in $V(rCy-y)$ have degree $3$ in $G$.
Since $|V(G)|\geq 7$ and $x_2w_2\in E(G)$, $rCy-y$ contains two adjacent vertices of degree $3$ in $G$, a contradiction.

\medskip

{\it Case} 2. $y\in V(w_2Cx_2)$. Then $d_G(w_2)=2$.

Consider $G_2=G-w_2$. Let $y'$ denote the unique vertex in $N_{G_2}(y)\cap N_{G_2}(w_1)$. Similar to Case 1, if $y'\notin\{ r,x_2\}$ then, by induction, $G_2-\{x_1,x_2,y,w_1\}$ has a Hamiltonian path $H_2$ between $r$ and $y'$. Hence, $H_2 + y'y$ is a Hamiltonian path between $r$ and $y$ in $G-\{x_1,x_2,w_1,w_2\}$. 

If $y'=r$, then let $x_2'$ be the neighbor of $x_2$ on $w_2Cx_2$; now $rx_2' \cup yCx_2'$ is a Hamiltonian path between $r$ and $y$ in $G-\{x_1,x_2,w_1,w_2\}$. 
If $y'=x_2$, then $x_2w_1,x_2y\in E(G)$. It follows that $d_G(r)=d_G(y)=3$, which gives a contradiction since $d_G(r,y)=2$.
\end{proof}

Next, we discuss results related to the counting idea from \cite{AAT2020}.
Let $S$ be an independent set in a $4$-connected planar triangulation $G$ and $F\subseteq E(G)$ consist of $|S|$ edges incident with $S$. Alahmadi \textit{et al.} \cite{AAT2020} observed that $G-F$ is not $4$-connected only if some vertex in $S$ is contained in a separating $4$-cycle, or some vertex in $S$ is adjacent to three vertices in a separating $4$-cycle, or two vertices in $S$ are contained in a separating $5$-cycle, or three vertices in $S$ occur in some diamond-$6$-cycle. A \textit{diamond-$6$-cycle} is a graph isomorphic to the graph shown on the left in Figure \ref{diamond cycle}, in which the vertices of degree $3$ are called \textit{crucial} vertices. (A {\it diamond-$4$-cycle} is a graph isomorphic to the graph shown on the right in Figure~\ref{diamond cycle}, where the two degree $3$ vertices not adjacent to the degree $2$ vertex are its {\it crucial} vertices.) We say that $S$ {\it saturates} a $4$-cycle or $5$-cycle $C$ in $G$ if $|S\cap V(C)|=2$, and  $S$ {\it saturates} a diamond-$6$-cycle $D$ in $G$ if $S$ contains three crucial vertices of $D$.
\begin{figure}[htb] 
	\begin{center}
        \begin{minipage}{.2\textwidth}
        		\resizebox{3cm}{!}{\begin{tikzpicture}[scale=1, Wvertex/.style={circle, draw=black, fill=white, scale=2}, rvertex/.style={circle, draw=black, fill=black, scale=0.2},bvertex/.style={circle, draw=black, fill=white, scale=0.2}]

\node [bvertex] (v1) at (90:1) {};
\node [bvertex] (v2) at (210:1) {};
\node [bvertex] (v3) at (330:1) {};

\draw (v1) -- (v2) -- (v3) -- (v1) ;

\node [rvertex] (v4) at (30: 0.5) {};
\node [rvertex] (v5) at (150: 0.5) {};
\node [rvertex] (v6) at (270: 0.5) {};

\node [rvertex] (u4) at (30: 1.0) {};
\node [rvertex] (u5) at (150: 1.0) {};
\node [rvertex] (u6) at (270: 1.0) {};

\draw (v1) -- (u5) -- (v2) -- (u6) -- (v3) -- (u4) -- (v1);

\draw (v4) -- (u4);
\draw (v5) -- (u5);
\draw (v6) -- (u6);

\end{tikzpicture}	}
        \end{minipage}
            \hspace{2cm}
        \begin{minipage}{.2\textwidth}
        		\resizebox{3.3cm}{!}{\begin{tikzpicture}[scale=1, Wvertex/.style={circle, draw=black, fill=white, scale=2}, rvertex/.style={circle, draw=black, fill=black, scale=0.2},bvertex/.style={circle, draw=black, fill=white, scale=0.2}]

\node [rvertex] (v0) at (0,0) {};
\node [bvertex] (v1) at (90:1) {};
\node [rvertex] (v2) at (180:1) {};
\node [bvertex] (v3) at (270:1) {};
\node [bvertex] (v4) at (0: 1) {};

\draw (v1) -- (v2) -- (v3) -- (v4) -- (v1);
\draw (v0) -- (v1);
\draw (v0) -- (v2);
\draw (v0) -- (v3);

\end{tikzpicture}	}
        \end{minipage}
    \end{center}
    \caption{diamond-$6$-cycle (left); diamond-$4$-cycle (right); solid vertices represent the crucial vertices.}
    \label{diamond cycle}
\end{figure}
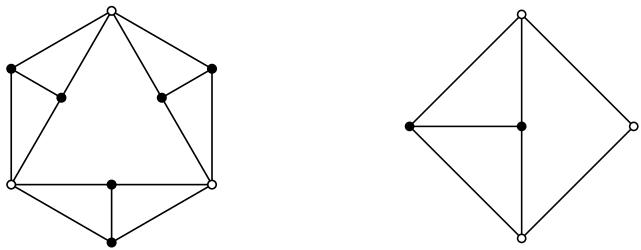

For a $5$-connected planar triangulation $G$, Alahmadi {\it et al.}\cite{AAT2020}  showed that there exists an independent set $S$ consisting of $\Omega(n)$ vertices of degree at most $6$ in $G$, such that $G-F$  is $4$-connected for each set $F$ consisting of $|S|$ edges of $G$ that are incident with $S$. Then it follows from a simple calculation that  $G$ has $2^{\Omega (n)}$ Hamiltonian cycles. 
Such  large independent sets need not exist in $4$-connected planar triangulations because of the existence of vertices of degree $4$ or separating $4$-cycles.

Next, we prove two lemmas that will help us deal with vertices of degree $4$ and separating $4$-cycles.
Let $G$ be a plane graph.
Suppose $u$ is a vertex of degree at most $6$ in $G$.
Define the {\it link} of $u$ in $G$, denoted by $A_u$, as
$$A_u = \begin{cases}
            E(G[N(u)]), & \textrm{ if $d(u)=4$,}\\
            \{e \in E(G): \textrm{$e$ is incident with $u$ and $G-e$ is $4$-connected}\}, & \textrm{ if $d(u)\in \{5,6\}$.}
        \end{cases}$$
\begin{lem}\label{lem:Au_size}
Let $G$ be a $4$-connected planar triangulation. Suppose $S$ is an independent set of vertices of degree at most $6$ in $G$ such that, for any $u\in S$ with $d(u)\in \{5,6\}$, no degree $4$ neighbors of $u$ are adjacent in $G$. Then the following statements hold:
\begin{itemize}
    \item [\textup{(i)}] For $u\in S$ with $d(u)\in \{5,6\}$, 
    $\{v\in N(u): uv\notin A_u\}$ is independent in $G$ 
    and, hence, $|A_u| \geq \lceil d(u)/2 \rceil$.
    \item [\textup{(ii)}] If $S$ saturates no $4$-cycle in $G$, then, for any distinct $u_1,u_2\in S$, $E(G[N[u_1]]) \cap E(G[N[u_2]]) = \emptyset$.
    
\end{itemize}
\end{lem}
\begin{proof}
Suppose $u \in S$ and $d(u)\in \{5,6\}$, and suppose there exist two edges $e_1=uv_1, e_2=uv_2 \in E(G)\backslash A_u$ 
with $v_1v_2\in E(G)$. Let $v_0, v_3 \in N(u)\backslash \{v_1, v_2\}$ be the neighbors of $v_1, v_2 $ in $G[N(u)]$, respectively. Since $G-e_1$ is not $4$-connected, there exists a vertex $z\in V(G)$ such that $\{z, v_0,v_2\}$ is a $3$-cut in $G-e_1$. Since $G$ is a planar triangulation, we have $zv_0,zv_2\in E(G)$. Since $G-e_2$ is not $4$-connected, we see from planarity that $\{z,v_1,v_3\}$ is a $3$-cut in $G-e_2$. Thus, $zv_1,zv_3\in E(G)$ as $G$ is a planar triangulation.
Since $G$ has no separating triangles, we have $d(v_1) = d(v_2) = 4$, a contradiction. Thus, (i) holds.

For (ii), suppose $S$ saturates no $4$-cycle in $G$, and let $u_1, u_2 \in S$ be distinct. Suppose there exists $e\in E(G[ N[u_1]]) \cap E(G[ N[u_2]])$. Since $S$ is an independent set, it follows that $u_1, u_2$, and the two vertices incident with $e$ form a $4$-cycle in $G$, contradicting the assumption that $S$ saturates no $4$-cycle in $G$. Hence $E(G[ N[u_1]]) \cap E(G[ N[u_2]])=\emptyset$.
\end{proof}

The following lemma is derived by using an idea similar to one in \cite{AAT2020}.

\begin{lem}\label{4conn}
 Let $G$ be a $4$-connected planar triangulation and $S$ be an independent set of vertices of degree at most $6$ in $G$. Suppose that $A_u\neq \emptyset$ for all $u\in S$, that $S$ saturates no $4$-cycle, or $5$-cycle, or diamond-$6$-cycle in $G$, and that no degree $4$ vertex of $G$ in $S$ has a neighbor of degree $4$ in $G$. Let $F\subseteq \bigcup_{u\in S} A_u$ with $|F\cap A_u|\leq 1$ for all $u\in S$. Then $G-F$ is $4$-connected.  
\end{lem}
\begin{proof}
Suppose there exists some $F\subseteq \bigcup_{u\in S} A_u$ such that $|F\cap A_u|\leq 1$ for all $u\in S$, and $G-F$ is not $4$-connected.
Let $K$ be a minimum cut of $G-F$; so $|K|\leq 3$. Let $G_1,G_2$ be subgraphs of $G-F$ such that $G-F=G_1\cup G_2$, $V(G_1\cap G_2)=K, E(G_1\cap G_2)=\emptyset$, and $V(G_i)\neq K$ for $i=1,2$. Let $F'$ be the set of the edges of $G$ between $G_1-K$ and $G_2-K$. Then $F'\subseteq F$ and $F'\neq \emptyset$ (as $G-K$ is connected).

\begin{obs}\label{obs1}
Since $G$ is a $4$-connected planar triangulation, for each $e\in F'$, the two vertices incident with $e$ have exactly two common neighbors, which must be contained in $K$.
\end{obs}
\begin{obs}\label{obs2}
For any two edges $e_1,e_2\in F'$,  there do not exist distinct vertices $u,v\in K$ such that all vertices incident with $e_1$ or $e_2$ are contained in $N_G(u) \cap N_G(v)$. For, otherwise, $G[N_G[u] \cup N_G[v]]$ contains a $4$-cycle with two vertices in $S$, contradicting the assumption that $S$ saturates no $4$-cycle in $G$.
\end{obs}

By Observation \ref{obs1}, $|K|\geq 2$. By Observations \ref{obs1} and \ref{obs2}, $|F'|\le  \binom{|K|}{2}$. Hence, $1\le |F'|\le 3$. Moreover, $|K|=3$ as, otherwise, $|K|=2$ and $|F'|\leq \binom{2}{2}=1$, contradicting the assumption that $G$ is $4$-connected. By the definition of $A_u$, if $e\in A_u \cap F'$ and $d_G(u)=4$, then $u\in K$; if $e\in A_u \cap F'$ and $d_G(u) \in \{5,6\}$, then $e$ is incident with $u$ and $u \notin K$.

Suppose $|F'|=1$ and let $e\in F'$ with $e\in A_u$ for some $u\in S$. If $d_G(u)=5$ or $6$, then $u$ is incident with $e$ and $G-e$ is $4$-connected by the definition of $A_u$, contradicting the fact that $K$ is a $3$-cut of $G-F'=G-e$. Thus $d_G(u)=4$ and $u\in K$. Let $e=w_1w_2$ and $K=\{u,v,w\}$ such that $w_1\in V(G_1)\backslash V(G_2)$, $w_2\in V(G_2)\backslash V(G_1)$, and  $N_G(w_1)\cap N_G(w_2)=\{u,v\}$. 
Again since $G$ is a planar triangulation and $K$ is a $3$-cut in $G-e$, we have $wu, wv\in E(G)$. Hence $C_1 = uw_1vwu$ and $C_2 = uw_2vwu$ are $4$-cycles in $G$. Let $x\in N_G(u)\backslash \{w, w_1, w_2\}$. Then $G[N_G(u)] = x w_2 w_1 w x$ or $G[N_G(u)] = x w_1 w_2 w x$. In the former case, $V(G_1)\backslash K=\{w_1\}$ as, otherwise, $\{w_1,w,v\}$ would be a $3$-cut in $G$; so $w_1$ and $u$ are two adjacent vertices of degree $4$ in $G$, a contradiction. In the latter case, $V(G_2)\backslash K=\{w_2\}$ as, otherwise, $\{w_2,w,v\}$
would be a $3$-cut in $G$; so $w_2$ and $u$ are two adjacent vertices of degree $4$ in $G$, a contradiction.

If $|F'|=2$ and let $F'=\{e_1,e_2\}$, then by Observations \ref{obs1} and \ref{obs2}, each vertex in $K$ is adjacent to both vertices incident with some edge in $F'$, and exactly one vertex of $K$ is adjacent to all vertices incident with $e_1$ or $e_2$. Hence, some $5$-cycle in the subgraph of $G$ induced by $K$ and the vertices incident with $F'$ contains two vertices from $S$, contradicting the assumption that $S$ saturates no $5$-cycle in $G$.  

Hence, $|F'|=3$, and let $e_1,e_2,e_3\in F'$ where $e_i\in A_{u_i}$ and $u_i\in S$ for $i=1,2,3$. Since $S$ is independent and saturates no $4$-cycle or 5-cycle, $F'$ is a matching in $G$. If two vertices in $\{u_1,u_2,u_3\}$ have degree $4$ in $G$, then these two vertices are contained in $K$ and in a $4$-cycle in $G$, a contradiction. If exactly one vertex in $\{u_1,u_2,u_3\}$, say $u_1$, has degree $4$ in $G$, then $u_1\in K$ and $u_1$ must be adjacent to a vertex in $\{u_2,u_3\}$,
contradicting the assumption that $S$ is independent. So $u_1, u_2$, and $u_3$ all have degree $5$ or $6$ in $G$. But then by Observations \ref{obs1} and \ref{obs2}, we see that the subgraph of $G$ induced by $K$ and the vertices of $G$ incident with $F'$ contains a diamond-$6$-cycle in which $u_1, u_2, u_3$ are three crucial vertices, contradicting the assumption that $S$ saturates no diamond-$6$-cycle. 
\end{proof}

We also need the following two lemmas from Lo \cite{Lo2020} and Alahmadi \textit{et al.} \cite{AAT2020}, that will help us to find an independent set saturating no $4$-cycle, or $5$-cycle, or diamond-$6$-cycle.

\begin{lem}[Lo \cite{Lo2020}] \label{5-cycle}
Let $G$ be a $4$-connected planar triangulation and let $S$ be an independent set of vertices of degree at most $6$ in $G$,  such that $S$ saturates no $4$-cycle in $G$. Then there exists a subset $S'\subseteq S$ of size at least $|S|/541$ such that $S'$ saturates no $5$-cycle in $G$.
\end{lem}
\begin{lem}[Alahmadi, Aldred, and Thomassen \cite{AAT2020}; Lo \cite{Lo2020}] \label{diamond}
Let G be a $4$-connected planar triangulation and let $S$ be an independent set of vertices of degree at most $6$ in $G$,  such that $S$ saturates no $4$-cycle in $G$. Then there exists a subset $S'\subseteq S$ of size at least $|S|/301$ such that $S'$ saturates no diamond-$6$-cycle in $G$.
\end{lem}

We need another result from Lo \cite{Lo2020}, which shows that any $4$-connected planar triangulation either has a large independent set saturating no $4$-cycle, or contains two vertices with many common neighbors.

\begin{lem}[Lo \cite{Lo2020}]\label {ISET}
Let $G$ be a $4$-connected planar triangulation. Let $S$ be an independent set of vertices of degree at most $6$ in $G$, and $S'$ be a maximal subset of $S$ such that $S'$ saturates no $4$-cycle in $G$. Then there exist  distinct vertices $v,x\in V(G)$ such that $|(N(v)\cap N(x))\cap S| \geq |S|/(9|S'|)$.
\end{lem}

The following result can be easily deduced from the previous three lemmas.

\begin{lem}\label {2cases}
Let $G$ be a $4$-connected planar triangulation on $n$ vertices.  Let $I$ be an independent set of vertices of degree at most $6$ in $G$. For any positive integer $t$, one of the following statements holds: 
\begin{itemize}
    \item[\textup{(i)}] There exist distinct vertices 
    $v,x\in V(G)$ such that $|N(v)\cap N(x)\cap I| \geq t$.
\item[\textup{(ii)}] There is a subset $S\subseteq I$, such that  $|S|>
|I|/(t\times 9 \times 541\times 301)$ and  $S$ saturates no $4$-cycle, or $5$-cycle, or diamond-$6$-cycle in $G$.
\end{itemize}
\end{lem}
\begin{proof}
Let $S_1$ be a maximal subset of $I$ such that $S_1$ saturates no $4$-cycle in $G$. If $|S_1|\leq |I|/(t\times 9)$, then by Lemma~\ref{ISET} there are distinct vertices $v,x$ in $G$ such that $|N(v)\cap N(x)\cap I|\geq |I|/(9|S_1|)\geq |I|/(9|I|/(t\times 9))=t$; so (i) holds. 

Now suppose $|S_1|> |I|/(t\times 9)$. By Lemmas \ref{5-cycle} and \ref{diamond}, there exists $S\subseteq S_1$  such that $S$ saturates no $4$-cycle, or $5$-cycle, or diamond-$6$-cycle in $G$, and $$|S|\geq |S_1|/(541\times 301)>|I|/(t\times  9\times 541\times 301);$$ 
thus (ii) holds.
\end{proof}

We conclude this section by stating two results on Tutte paths and Tutte cycles.
 Let $G$ be a graph and $H\subseteq G$. An {\it $H$-bridge} of $G$ is a subgraph of $G$ induced by either an edge in $E(G)\backslash E(H)$ with both incident vertices in $V(H)$, or all edges in $G-H$ with at least one incident vertex in a single component of $G-H$.
For an  $H$-bridge $B$ of $G$, the vertices in $V(B\cap H)$ are the {\it attachments} of $B$ on $H$.
A path (or cycle) $P$ in a graph $G$ is called a {\it Tutte path} (or {\it Tutte cycle}) if every $P$-bridge of $G$ has at most three attachments on $P$. If, in addition, every $P$-bridge of $G$ containing an edge of some subgraph $C$ of $G$ has at most two attachments on $P$, then $P$ is called a {\it $C$-Tutte path}  (or {\it $C$-Tutte cycle}) in $G$.
Thomassen \cite{Thomassen1983} proved the following result on Tutte paths in $2$-connected planar graphs.

\begin{lem}[Thomassen \cite{Thomassen1983}] \label{tuttepath}
Let $G$ be a $2$-connected plane graph and $C$ be its outer cycle, and let $x\in V(C)$,  $y\in V(G)\backslash\{x\}$, and $e\in E(C)$. Then $G$ has a $C$-Tutte path $P$ between $x$ and $y$ such that $e\in E(P)$. 
\end{lem}

Note that Lemma~\ref{tuttepath} implies that every $4$-connected planar graph is Hamiltonian connected and has a Hamiltonian cycle through two given edges that are cofacial.

A \textit{circuit graph} is an ordered pair $(G, C)$ consisting of a $2$-connected plane graph $G$ and a facial cycle $C$ of $G$
such that, for any $2$-cut $U$ of $G$, each component of $G- U$ contains a vertex of $C$. Jackson and Yu \cite{Jackson-Yu2002} showed that every circuit graph $(G, C)$ has a $C$-Tutte
cycle through a given edge of $C$ and two other vertices.

\begin{lem}[Jackson and Yu \cite{Jackson-Yu2002}]\label{2vts}
Let $(G,C)$ be a circuit graph, and let $r,z$ be vertices of $G$ and $e\in E(C)$. Then $G$ contains a $C$-Tutte cycle $H$ such that $e\in E(H)$ and $r,z\in V(H)$.
\end{lem}

\section{Quadratic bound}
We start with a technical lemma for finding distinct Hamiltonian cycles.
Recall the link $A_u$ for a vertex $u$ of degree at most $6$ in a plane graph. 
\begin{lem}\label{lem:edge-link}
Let $G$ be a $4$-connected plane graph and $e\in E(G)$. Suppose $u$ is a vertex of degree at most $6$ in $G$ such that $G[N(u)]$ is a cycle and $e\notin E(G[N[u]])$. 
Moreover, assume that if $d(u)\in \{5,6\}$ then 
$\{v\in N(u): uv\notin A_u\}$ is an independent set in $G$, and that if $d(u) = 4$ then there exist two nonadjacent neighbors of $u$ each having degree at least $5$ in $G$. 
Then the following statements hold.

\begin{itemize}
    \item [\textup{(i)}] $G$ has a Hamiltonian cycle through  $e$ as well as an edge in $A_u$.
    \item [\textup{(ii)}] For any $y\in V(G)\backslash\{u\}$ cofacial with $e$ but not incident with $e$, $G-y$ has a Hamiltonian cycle through $e$ and an edge in $A_u$ not incident with $y$.
    \end{itemize}
\end{lem}
\begin{proof} 
Let $y\in V(G)\backslash \{u\}$ be cofacial with $e$ but not incident with $e$. Consider a drawing of $G$ in which $y$ is contained in the infinite face of $G-y$. Let $C$ denote the facial cycle of $G$ containing $e$ and $y$, and $C'$ denote the outer cycle of $G-y$.
Then $e\in E(C')$ as $y$ is cofacial with $e$ and not incident with $e$. Since $G$ is $4$-connected, both $(G,C)$ and $(G-y, C')$ are circuit graphs.

\medskip

{\it Case} 1. $d_G(u)\in\{5,6\}$ and $d_G(u)-|A_u|\leq 1$.

By Lemma~\ref{2vts}, $G$ has a $C$-Tutte cycle $D$ through $e$, and $G-y$ has a $C'$-Tutte cycle $D_1$ through $e$. Since $G$ is $4$-connected, $D$ is a Hamiltonian cycle in $G$, and $D_1$ is a Hamiltonian cycle in $G-y$. Since $d_G(u)-|A_u|\leq 1$, both $D$ and $D_1$ contain some edge in $A_u$. Thus, (i) and (ii) hold. 

\medskip

{\it Case} 2. $d_G(u)\in\{5,6\}$ and $d_G(u)-|A_u|=2$.

 Then let $r\in N_G(u)$ such that $ur\notin A_u$, and let $G':=G-ur$.
 
 Let $C_1$ be a facial cycle of $G'$ containing $e$. Since $G$ is $4$-connected, $(G', C_1)$ is a circuit graph. By Lemma~\ref{2vts}, $G'$ has a $C_1$-Tutte cycle $D_1$ through $e$, $r$, and $u$, which is a Hamiltonian cycle in $G$ containing $e$. Since $d_G(u)-|A_u|=2$ and $ur\notin A_u$, $D_1$ must contain an edge in $A_u$. Hence, (i) holds. 
 
 Let $C_2$ be the outer cycle of $G'-y$. Since $G$ is $4$-connected, $(G'-y, C_2)$ is a circuit graph. By Lemma~\ref{2vts}, $G'-y$ has a $C_2$-Tutte cycle $D_2$ through $e$ and every vertex in $\{r,u\}\backslash \{y\}$. Now $D_2$ is a Hamiltonian cycle in $G-y$ containing $e$ as $G$ is $4$-connected. Moreover, $D_2$ contains an edge in $A_u$ as $d_G(u)-|A_u|= 2$ and $ur\notin A_u$. Hence, (ii) holds.
\medskip

{\it Case} 3.  $d_G(u)\in\{5,6\}$ and $d_G(u)-|A_u|\ge 3$.

Since $\{v\in N_G(u): uv\notin A_u\}$ is an independent set in $G$, $|A_u|\ge \lceil d_G(u)/2\rceil=3$ (as $d_G(u)\in\{5,6\}$). Hence, $d_G(u)=6$ and $|A_u|=3$ (as $d_G(u)-|A_u|\ge 3$).
Let $r_1,r_2\in N_G(u)$ such that $ur_1,ur_2\notin A_u$, and let $G'=G-\{ur_1,ur_2\}$. Note that $r_1r_2\notin E(G)$.

Let $C_1$ be a facial cycle of $G'$ containing $e$. Then $(G', C_1)$ is a circuit graph as $G$ is $4$-connected and $d_G(u)=6$. It follows from Lemma~\ref{2vts} that $G'$ has a $C_1$-Tutte cycle $D_1$ through $e$, $r_1$, and $r_2$. If $D_1$ is a Hamiltonian cycle in $G$, then (i) holds, since $D_1$ contains an edge of $A_u$ (because $d_G(u)-|A_u|$=3 and $ur_1,ur_2\notin A_u)$. So suppose $V(G)\backslash V(D_1)\ne \emptyset$. 
Then there exists a $D_1$-bridge of $G'$, say $B$, such that $V(B)\backslash V(D_1)\ne \emptyset$ and $V(B\cap D_1) \le 3$. Observe that $u\in V(B)\backslash V(D_1)$ otherwise, $V(B\cap D_1)$ is a $3$-cut in $G$, a contradiction. Since  $G[N_G(u)]$ is a cycle and $r_1r_2\notin E(G)$, $V(B)\cap V(r_1D_1r_2-\{r_1,r_2\})\ne \emptyset$ and $V(B)\cap V(r_2D_1r_1-\{r_1,r_2\})\ne \emptyset$. Thus, since $|V(B)\cap V(D_1)|\le 3$, we may assume $V(B)\cap V(r_2D_1r_1-\{r_1,r_2\})=\{z\}$. Now since $d_G(u)=6$, $\{u,z\}$ is a $2$-cut in $G$, or $\{u\}\cup (V(B\cap D_1)\backslash \{z\})$ is a $3$-cut in $G$, a contradiction.

Let $C_2$ be the outer cycle of $G'-y$. Since $G$ is $4$-connected and $d_G(u)=6$, $(G'-y, C_2)$ is a circuit graph. By Lemma~\ref{2vts}, $G'-y$ has a $C_2$-Tutte cycle $D_2$ through $e$ and every vertex in $\{r_1,r_2\}\backslash \{y\}$. Similarly, we can show that $D_2$ is a Hamiltonian cycle in $G-y$ containing $e$ as $G$ is $4$-connected and $D_2$ is a $C_2$-Tutte cycle. (In particular, note that if $B$ is a $D_2$-bridge and $V(B)\backslash V(D_2)$ contains a neighbor of $y$, then  $|V(B)\cap V(D_2)|\le 2$.) Moreover, $D_2$ contains an edge in $A_u$ as $|A_u|= 3$ and $ur_1,ur_2\notin A_u$. Hence, (ii) holds.

\medskip

{\it Case} 4. $d_G(u)=4$.

By symmetry, let $G[N_G(u)] = x_1x_2x_3x_4x_1$. By our assumption on $u$, two nonadjacent neighbors of $u$ must each have degree at least $5$ in $G$. Without loss of generality, assume that $d_G(x_{2})\geq 5$ and $d_G(x_4)\geq 5$.

We claim that $(G-u)+x_1x_3$ or $(G-u)+x_2x_4$ is $4$-connected. For, suppose  $(G-u)+x_1x_3$ is not $4$-connected, and let $S$ be a $3$-cut in $(G-u)+x_1x_{3}$. Then $\{x_1, x_{3}\}\subseteq S$, and $S\cup \{u\}$ is a $4$-cut in $G$ separating $x_{2}$ and $x_{4}$. Suppose $G_1, G_2$ are the components of $G-(S\cup \{u\})$ containing $x_2, x_4$, respectively. 
  Since $d_G(x_{2i})\geq 5$ for $i=1,2$, $|V(G_i)|\geq 2$ for $i=1,2$. 
  Let $w_i \in V(G_i) \backslash \{x_{2i}\}$ for $i=1,2$.
  Since $G$ is $4$-connected, there exist a path $Q_i'$ from $w_i$ to $x_1$ in $(G-(S \backslash \{x_1\})) - x_{2i}$ and a path $Q_i''$ from $w_i$ to $x_3$ in $(G-(S \backslash \{x_3\})) - x_{2i}$. 
  Observe that $V(Q_i'\cup Q_i'') \subseteq (V(G_i)\backslash \{x_{2i}\}) \cup\{x_1,x_3\}$. Hence, $G-\{u,x_2,x_4\}$ has two internally disjoint paths between $x_1$ and $x_3$. This implies that $(G-u)+x_2x_4$ is $4$-connected. 

So without loss of generality assume that
$G^* = (G-u)+x_1x_3$ is  $4$-connected and that the edge $x_1 x_3$ is inside the face of $G-u$ bounded by $x_1 x_2 x_3 x_4 x_1$. Let $G'$ be the plane graph obtained from $G^*$ by inserting two vertices $r$ and $z$ into the faces of $G^*$ bounded by $x_1x_2x_3x_1$ and $x_1x_3x_4x_1$, respectively, and then adding edges $r x_i$ for $i=1,2,3$ and $zx_i$ for $i=1,3,4$.

Since $G^*$ is $4$-connected, $(G',C)$ is a circuit graph.
By Lemma~\ref{2vts}, $G'$ has a $C$-Tutte cycle $D'$ containing $e, r$, and $z$, which is a Hamiltonian cycle in $G'$ as $G^*$ is $4$-connected. It is easy to check that $D'$ can be modified at $r$ and $z$ to give a Hamiltonian cycle in $G$ containing $e$ and an edge in $A_u$; so (i) holds. 

To prove (ii), we apply Lemma~\ref{2vts} to the circuit graph $(G'-y, C_1)$, where $C_1$ is the outer cycle of $G'-y$ containing $e$. Then $G'-y$ has a $C_1$-Tutte cycle $D_1$ through $e$, $r$, and $z$. (Note that if $uy\in E(G)$, then $\{r,z\}\cap V(C_1)\ne \emptyset$.) Since $G^*$ is $4$-connected, $D_1$ is a Hamiltonian cycle in $G'-y$.
It is straightforward to check that $D_1$ can be modified to give a Hamiltonian cycle in $G-y$ containing $e$ and an edge in $A_u$ not incident with $y$.
\end{proof}

We now prove that in a $4$-connected planar triangulation on $n$ vertices, any two cofacial edges are contained in $\Omega(n)$ Hamiltonian cycles.

\begin{lem}\label{2edge} 
Let $n$ be an integer with $n\geq 4$, $G$ be a $4$-connected planar triangulation on $n$ vertices, $T$ be a triangle in $G$, and $e_1,e_2\in E(T)$. Then $G$ contains at least $c_1 n$ Hamiltonian cycles through $e_1$ and $e_2$, where $c_1=(12\times 63\times 541 \times 301)^{-1}$.
\end{lem}
\begin{proof}
We apply induction on $n$. 
Since  $G$ is a $4$-connected plane graph and $e_1, e_2$ are cofacial in $G$, it follows from Lemma~\ref{tuttepath} that $G$ has a Hamiltonian cycle through $e_1$ and $e_2$.
So the assertion holds when $n\leq 1/c_1$.
Now assume $n> 1/c_1$ and the assertion holds for $4$-connected planar triangulations on fewer than $n$ vertices.

Consider a drawing of $G$ in which $T$ is its outer cycle. Let $y\in V(T)$ be  incident with both $e_1$ and $e_2$, and let $e_3$ be the edge in $E(T)\backslash \{e_1, e_2\}$.

We may assume that if there exist two adjacent vertices $u_1,u_2$ in $G$ with $d_G(u_1) = d_G(u_2) = 4$, then $u_1u_2=e_3$ or $y\in \{u_1,u_2\}$.
For, suppose there exist $u_1,u_2\in V(G)\backslash\{y\}$ such that $d_G(u_1)=d_G(u_2)=4$ and $u_1u_2\neq e_3$. We contract the edge $u_1u_2$ to obtain a planar triangulation $G^*$ on $n-1$ vertices. (We retain the edges $e_1$ and $e_2$.) Note that $G^*$ is  $4$-connected (as $n> 1/c_1>6$) and $T$ is a triangle in $G^*$. So by induction, $G^*$ has $c_1(n-1)$ Hamiltonian cycles through $e_1$ and $e_2$. Observe that all such cycles in $G^*$ can be modified to give $c_1(n-1)$ distinct Hamiltonian cycles in $G$ through the edges $e_1,e_2,$ and $u_1 u_2$. Therefore, it suffices to show that $G$ has a Hamiltonian cycle through $e_1$ and $e_2$ but not $u_1u_2$, as $c_1(n-1)+1\geq c_1n$. So let $C_1$ denote the outer cycle of $G_1:=(G-y)-u_1u_2$. Observe that $e_3\in E(C_1)$ and $(G_1, C_1)$ is a circuit graph as $G$ is $4$-connected and planar. By Lemma~\ref{2vts}, $G_1$ contains a $C_1$-Tutte cycle $H_1$ through $e_3, u_1$, and $u_2$. Moreover, $H_1$ is a Hamiltonian cycle in $G_1$ (since $G$ is $4$-connected). Therefore, $(H_1- e_3)+ \{y,e_1,e_2\}$ is a Hamiltonian cycle in $G$ through $e_1, e_2$ and avoiding $u_1u_2$.

Since $G$ has minimum degree at least $4$ and $|E(G)|=3n-6$ by Euler's formula, we have 
\begin{align*}
    2(3n-6) = 2|E(G)| & = \displaystyle\sum_{\{v\in V(G):4\leq d(v)\leq 6\}} d(v) + \displaystyle\sum_{\{v\in V(G): d(v)\geq 7\}} d(v) \\
    & \geq 4|\{v \in V(G):d(v)\leq 6\}| + 7(n-|\{v \in V(G):d(v)\leq 6\}|).
\end{align*}
It follows that $|\{v \in V(G):d(v)\leq 6\}| \geq n/3 +4$. By the Four Color Theorem, there exists an independent set $I$ of vertices of degree at most 6 in $G$ with $I\cap V(T)=\emptyset$ and $|I|\geq (n/3+4-3)/4 \geq n/12$.  By Lemma~\ref{2cases} (with $t=7$), either there exist distinct $v, x \in V(G)$ such that $|N(v)\cap N(x)\cap I|\geq 7$, or there is a subset $S\subseteq I$ such that $|S|> c_1 n$ and $S$ saturates no $4$-cycle, or $5$-cycle, or diamond-$6$-cycle in $G$. Moreover, $S\cap V(T)=\emptyset$ as $I\cap V(T)=\emptyset$.

\medskip

{\it Case} 1. There exist distinct $v, x \in V(G)$ such that $|N(v)\cap N(x)\cap I|\geq 7$.

Recall that any two adjacent degree $4$ vertices of $G$ cannot be contained in $V(G)\backslash V(T)$.
Since $|N(v)\cap N(x)\cap I|\geq 7$, $G$ has at least two separating $4$-cycles $D_1$ and $D_2$, such that $|V(\overline{D_i})|\geq 6$ for $i=1,2$, and $\overline{D_1}-{D_1}$ and $\overline{D_2}-{D_2}$ are disjoint. Without loss of generality, we may assume  $|V(\overline{D_1}-{D_1})|\leq n/2$. By our assumptions on $G$ and applying Lemma~\ref{lem:2Hamiltonpaths}, we see that $\overline {D_1}-(V(D_1)\backslash \{a,b\})$ has at least two Hamiltonian paths between $a$ and $b$ for any distinct $a, b\in V(D_1)$.

Let $G_1^*$ be obtained from $G$ by contracting $\overline{D_1}-D_1$ to a new vertex $v_1$. Observe that $G_1^*$ is a $4$-connected planar triangulation with outer cycle $T$. It follows by induction that $G_1^*$ has at least $c_1(n-|V(\overline{D_1}-D_1)|+1)\geq c_1n/2$ Hamiltonian cycles through $e_1$ and $e_2$. 
For each of such Hamiltonian cycles in $G_1^*$, say $H^*$, let $a_1,b_1\in N_{G_1^*}(v_1)$ such that $a_1v_1b_1\subseteq H^*$. 
We can then form a Hamiltonian cycle in $G$ through $e_1$ and $e_2$  by taking the union of $H^*-v_1$ and a Hamiltonian path between $a_1$ and $b_1$ in $\overline{D_1}-(V(D_1)\backslash \{a_1,b_1\})$. Thus $G$ has at least $2(c_1n/2)=c_1n$ Hamiltonian cycles through $e_1$ and $e_2$.

\medskip

{\it Case} 2. There is an independent set $S$ of vertices of degree at most $6$ in $G$ such that $|S|> c_1 n$, $S\cap V(T) = \emptyset$, and 
$S$ saturates no $4$-cycle, or $5$-cycle, or diamond-$6$-cycle in $G$. 


 If there exist distinct $u_1,u_2\in S$ such that $|N_G(u_i)\cap V(T)|\geq 2$ for $i\in [2]$, then $u_1,u_2$ are contained in a $4$-cycle or a $5$-cycle in $G$, a contradiction. Hence, at most one vertex in $S$, say $x$, is adjacent to two vertices in $V(T)$. Let $S' = S$ if $x$ does not exist, and $S'= S\backslash \{x\}$ if $x$ exists. 
 Hence, $|S'|\ge |S|-1$ and for all $u\in S'$, $|N_G(u)\cap V(T)|\le 1$.
 
 Next we show that $S'$ satisfies the conditions of Lemma~\ref{lem:Au_size} and Lemma~\ref{4conn}.
 First suppose $d_G(y)> 4$. If two degree $4$ vertices are adjacent in $G$, they must be the two vertices in $V(T)\backslash\{y\}$. Hence, for any $u\in S'$, since $|N_G(u)\cap V(T)|\leq 1$, if $d_G(u)=4$ then $u$ is not adjacent to a degree $4$ vertex in $G$, and if $d_G(u)\in \{5,6\}$ then no degree $4$ neighbors of $u$ are adjacent in $G$. Now assume that $d_G(y)=4$. Notice that for any $v\in N_G(y)$, $|N_G(v)\cap V(T)|=2$, and thus, 
 $N_G(y)\cap S'=\emptyset$. Hence, for any $u\in S'$, if $d_G(u)=4$ then $u$ is adjacent to no degree 4 vertex in $G$, and if $d_G(u)\in \{5,6\}$ then no degree $4$ neighbors of $u$ are adjacent in $G$. Therefore, $S'$ satisfies the conditions of Lemma \ref{lem:Au_size} and Lemma~\ref{4conn}.

Let $k:=|S'|$ and $S'=\{u_1,u_2,\ldots, u_k\}$. Recall the definition of $A_{u_i}$ for $i\in [k]$,
and let
$A_i:=A_{u_i}\backslash \{ e\in E(G): \text{$e$ is incident with $y$}\}.$
By Lemma~\ref{lem:Au_size}, $E( G[N_G[u_i]]) \cap E(G[ N_G[u_j]])=\emptyset$ for $i\neq j$, and if $d_G(u_i)\in \{5,6\}$, $\{v\in N_G(u_i): vu_i\notin A_{u_i}\}$ is independent in $G$; so $|A_{u_i}|\ge 3$ for all $u_i\in S'$.
Hence, $|A_i|\geq 2$ for all $i\in[k]$.
Note that $e_3\notin E( G[N_G[u_i]])$, as $S'\cap V(T)=\emptyset$ and $|N_G(u_i) \cap V(T)|\leq 1$. 
 We now find $k+1 > c_1n $ Hamiltonian cycles $H_1, \ldots, H_{k+1}$ in $G$, as follows.

Let $F_0 = \emptyset$ and $X_1:=G-F_0=G$. Note that $e_3\notin E(G[N_G[u_1]])$ and by Lemma~\ref{lem:Au_size}, $u_1$ satisfies the conditions of Lemma~\ref{lem:edge-link} (with $u_1,e_3,X_1$ as $u, e, G$ in Lemma~\ref{lem:edge-link}, respectively). Since $y$ is cofacial with $e_3$ but not incident with $e_3$, it follows from (ii) of Lemma~\ref{lem:edge-link} that $X_1-y=G-y$ has a Hamiltonian cycle $D_1$ through $e_3$ and an edge $f_1\in A_1$. Hence, $H_1=(D_1-e_3)+ \{y, e_1,e_2\}$ is a Hamiltonian cycle in $G$ through $e_1,e_2,$ and $f_1\in A_1$. Set $F_1=\{f_1\}$.

Suppose for some $j\in [k+1]$ ($j\geq 2$) we have found an edge set $F_{j-1}=\{ f_1, \ldots, f_{j-1}\}$ where $f_i\in A_i$ for each $i \in [j-1]$, and a Hamiltonian cycle $H_l$ in $X_{l}:=G-F_{l-1}$ for each $l\in [j-1]$, such that $\{e_1, e_2, f_l\} \subseteq E(H_l)$ and $F_{l-1} \cap E(H_{l}) = \emptyset$. 
Consider the graph $X_{j}:=G-F_{j-1}$. By Lemma~\ref{4conn}, $X_{j}$ is $4$-connected. When $j=k+1$, $X_{k+1}:=G-F_{k}$ is $4$-connected; so by Tutte's theorem, $X_{k+1}$ has a Hamiltonian cycle $H_{k+1}$ through $e_1$ and $e_2$. We stop this process and output the desired $H_1, \ldots, H_{k+1}$.
Now suppose $j\leq k$.
Note that $G[N_G[u_j]]$ is a subgraph of $X_j$ (as $E(G[N_G[u_j]])\cap E(G[N_G[u_l]])=\emptyset$ for any $l\in [j-1]$), and that $e_3\in E(X_j)\backslash E( G[N_G[u_j]])$. We now show that $u_j$ satisfies the conditions of Lemma~\ref{lem:edge-link} (with $u_j, e_3, X_j$ as $u, e, G$ in Lemma~\ref{lem:edge-link}, respectively). Since $u_j\in S'$, $X_j-f$ is $4$-connected for any $f\in A_{u_j}$ (by Lemma~\ref{4conn}), and the link of $u_j$ in $X_j$ is same as $A_{u_j}$, as $G[N_G[u_j]]\subseteq X_i \subseteq G$. Hence, if $d_{X_j}(u_j)=d_G(u_j)\in \{5,6\}$, by Lemma~\ref{lem:Au_size}, 
$\{v\in N_{x_j}(u_j): vu_j\notin A_{u_j}\}=\{v\in N_G(u_j): vu_j\notin A_{u_j}\}$ is independent in $G$ (also independent in $X_j$);
if $d_{X_j}(u_j)=d_G(u_j)=4$, then all neighbors of $u_j$ each have degree at least $5$ in $X_j$, as $A_{u_j}=E(G[N_G(u_j)])$ and $X_j-f$ is $4$-connected for any $f\in A_{u_j}$. Therefore, by (ii) of Lemma~\ref{lem:edge-link}, $X_j-y$ has a Hamiltonian cycle $D_j$ through $e_3$ and some edge $f_j\in A_j$. Now $H_j=(D_j-e_3)+\{y,e_1,e_2\}$ is a Hamiltonian cycle in $G$ such that $\{e_1,e_2,f_j\}\subseteq E(H_j)$. Note $F_{j-1}\cap E(H_j)=\emptyset$ as $D_j\subseteq X_j$. Set $F_j=F_{j-1}\cup \{f_j\}$.

 Therefore, $G$ has at least $k+1=|S'|+1> c_1 n$ Hamiltonian cycles through $e_1, e_2$.
\end{proof}

\begin{proof}[Proof of Theorem~\ref{main1}] Let $c_2:=(12\times 90\times 541 \times 301)^{-1}$ and $c=c_2^2/2$. We show that every $4$-connected planar triangulation on $n$ vertices has at least $c n^2$ Hamiltonian cycles. It is easy to check that the assertion holds when $n\leq 1/ \sqrt{c}=\sqrt{2}/{c_2}$ as every $4$-connected planar graph is Hamiltonian by Tutte's theorem.
Hence we may assume that $n>\sqrt{2}/{c_2}$ and that the assertion holds for $4$-connected planar triangulations on fewer than $n$ vertices.

\medskip

{\it Case} 1. $G$ contains two adjacent vertices of degree $4$.

Let $u_1, u_2 \in V(G)$ such that $u_1 u_2 \in E(G)$ and $d_G(u_1) = d_G(u_2) = 4$.
Let $G^*$ be the graph obtained from $G$ by contracting the edge $u_1u_2$ to a new vertex $u^*$. By induction, $G^*$ has at least $c (n-1)^2$ Hamiltonian cycles from which we obtain at least $ c(n-1)^2$ Hamiltonian cycles in $G$ through the edge $u_1u_2$. 

Let $x_1,x_2,x_3,x_4$ be the vertices that occur on $G[N_{G^*}(u^*)]$ in the clockwise order such that $N_G(u_1)\cap N_G(u_2)=\{x_2,x_4\}$. Consider $G^*-u^*$. Note that  $G': = (G^*-u^*)+x_1x_3$ or $G'':=(G^*-u^*)+x_2x_4$ is a $4$-connected planar triangulation on $n-2$ vertices. (We add the edge $x_1 x_3$ or $x_2 x_4$ inside the face of $G^* - u^*$ containing $u^*$.)

First, suppose $G'$ is a $4$-connected planar triangulation. Note that $x_1x_2x_3x_1, x_1x_4x_3x_1$ are two triangles in $G'$. By Lemma~\ref{2edge}, $G'$ has at least $c_1(n-2)$ Hamiltonian cycles through $x_1x_{2i}$ and $x_{2i}x_3$ for each $i\in [2]$. Note that if $H$ is a Hamiltonian cycle in $G'$ through $x_1x_{2i}$ and $x_{2i}x_3$, then $(H- x_{2i}) \cup x_1u_1x_{2i}u_2x_3$ is a Hamiltonian cycle in $G-u_1 u_2$. Therefore, $G$ contains at least $2c_1(n-2)$ Hamiltonian cycles all  avoiding the edge $u_1u_2$. Hence, there exist at least $c (n-1)^2+2c_1(n-2)\geq c n^2$ Hamiltonian cycles in $G$ (as $n>\sqrt{2}/{c_2}$). 

Now assume that $G''$ is a $4$-connected planar triangulation. By Lemma~\ref{2edge}, for each $i\in [2]$, $G''$ contains at least $c_1(n-2)$ Hamiltonian cycles through $x_1x_{2i}$ and $x_2x_4$ (since $x_1x_2x_4x_1$ is a triangle in $G''$). Thus, $G^*$ has at least $2c_1(n-2)$ Hamiltonian cycles  through $x_2u^*, x_4u^*$. For any such cycle $H^*$, $(H^*-u^*) \cup x_2u_1u_2x_4$ and $(H^*-u^*) \cup x_2u_2u_1x_4$ are two distinct Hamiltonian cycles in $G$. Since $G^*$ has at least $c(n-1)^2$ Hamiltonian cycles among which at least $2c_1(n-2)$ such cycles contain $x_2u^*x_4$, it follows that $G$ has at least $c (n-1)^2+2c_1(n-2)\geq c n^2$ Hamiltonian cycles. 

\medskip

{\it Case} 2. No two vertices of degree $4$ in $G$ are adjacent.

Recall that $G$ contains an independent set $I$ of vertices of degree at most $6$ with $|I|\geq n/12$. By Lemma~\ref{2cases} (with $t=10$),  either there exist distinct vertices $v, x\in V(G)$ such that $|N(v)\cap N(x)\cap I|\geq 10$, or $G$ contains  $S\subseteq I$ such that  $|S|\geq c_2n$ and $S$ saturates no $4$-cycle, or $5$-cycle, or diamond-$6$-cycle in $G$. 

Suppose the former case holds. Since no two vertices of degree $4$ in $G$ are  adjacent, we can find a  separating $4$-cycle $D$ such that
$1<|V(\overline{D}-D)|\leq n/4$.
We contract $\overline{D}-D$
to a vertex and denote this new graph by $G_0$. Note that $G_0$ is a $4$-connected planar triangulation with $3n/4\leq |V(G_0)|=n-|V(\overline{D}-D)|+1\leq n-1$; so $G_0$ has at least $c(3n/4)^2$ Hamiltonian cycles by induction. Therefore, $G$ has at least $2c(3n/4)^2\geq c n^2$ Hamiltonian cycles, as  by Lemma~\ref{lem:2Hamiltonpaths}, $\overline{D}-(V(D)\backslash \{a,b\})$ has at least two Hamiltonian paths between $a$ and $b$ for any distinct $a,b\in V(D)$. 

Now assume that there exists an independent set $S$ of vertices of degree at most $6$ in $G$ with $|S|\geq c_2n$ such that $S$ saturates no $4$-cycle, or $5$-cycle, or diamond-$6$-cycle in $G$. By our assumptions on $G$, $S$ satisfies the conditions in Lemma \ref{lem:Au_size} and Lemma \ref{4conn}. Let $S=\{u_1, u_2 \ldots , u_k\}$. Recall the definition of $A_{u_i}$, the link of $u_i$ for each $i\in [k]$.

Let $F=\{f_1, f_2,\ldots, f_k\}$, where $f_i\in A_{u_i}$ for $i\in [k]$. Let $F_j=\{f_1,\ldots, f_j\}$ for each $j \in [k]$ and let $F_0=\emptyset$.

We claim that for each integer $j\in [k]$, there exists a collection of Hamiltonian cycles in $G$, say ${\cal C}_j$,
such that $|{\cal C}_j|=k-j+1$ and every cycle in ${\cal C}_j$ contains $f_j$ but no edge from $F_{j-1}$.
For each $j\in [k]$, let $X_j:=G-F_{j-1}$. By Lemma~\ref{4conn}, $X_j$ is $4$-connected for each $j$. If $j=k$, it follows from Lemma~\ref{tuttepath} that $X_k$ has a Hamiltonian cycle $H_{k+1}^{(k)}$ through the edge $f_k$. Moreover, $F_{k-1}\cap E(H_{k+1}^{(k)})=\emptyset$ as $H_{k+1}^{(k)}\subseteq X_k$. Let ${\cal C}_k=\{H_{k+1}^{(k)}\}$.

Now assume $j<k$. Let $F_j^{(j)}=\emptyset$ and $Y_{j+1}^{(j)}:=X_j-F_j^{(j)}=X_j$. Note that $f_j\in X_j$ as $X_j=G-F_{j-1}$. 
Note that $u_{j+1}$ satisfies the conditions in Lemma~\ref{lem:edge-link} (with $u_{j+1}, f_j, X_j$ as $u, e, G$ in Lemma~\ref{lem:edge-link}, respectively), since $u_{j+1}\in S$, $S$ satisfies the conditions in Lemmas~\ref{lem:Au_size} and \ref{4conn}, and $F_{j-1}\subseteq \cup _{i=1}^{j-1} A_{u_i}$ if $j\ge 1$. Then by (i) of Lemma~\ref{lem:edge-link}, $X_j$ has a Hamiltonian cycle $H_{j+1}^{(j)}$ containing $f_j$ and some edge $f_{j+1}^{(j)}\in A_{u_{j+1}}$. Set $F_{j+1}^{(j)}=\{f_{j+1}^{(j)}\}$.
For $j+2\leq l \leq k+1$,  suppose we have found an edge set $F_{l-1}^{(j)}=\{f_{j+1}^{(j)}, \ldots, f_{l-1}^{(j)}\}$, where $f_{t}^{(j)}\in A_{u_t}$ for $j+1\leq t\leq l-1$, such that, for each $j+1\leq t\leq l-1$, $Y_t^{(j)}:=X_j-F_{t-1}^{(j)}$ is $4$-connected and has a Hamiltonian cycle $H_t^{(j)}$ through $f_j$ and $f_t^{(j)}$. Consider $Y_l^{(j)}:=X_j-F_{l-1}^{(j)}$. Then $Y_l^{(j)}$ is $4$-connected by Lemma~\ref{4conn}. If $l=k+1$, then, by Lemma~\ref{tuttepath}, $Y_{k+1}^{(j)}:=X_j-F_{k}^{(j)}$ has a Hamiltonian cycle $H_{k+1}^{(j)}$ through $f_j$ and $F_{k}^{(j)}\cap H_{k+1}^{(j)}=\emptyset$. We stop the process and output the desired ${\cal C}_j=\{H_{j+1}^{(j)}, H_{j+2}^{(j)}, \ldots, H_{k+1}^{(j)}\}$. Now assume that $l< k+1$. Then $G[N_G[u_{l}]]$ is a subgraph of $Y_{l}^{(j)}$, and $u_{l}$ in $Y_{l}^{(j)}$ satisfies the conditions in Lemma \ref{lem:edge-link}. Since $f_j\in E(Y_{l}^{(j)})\backslash E(G[N_G[u_{l}]])$, we apply (i) of Lemma~\ref{lem:edge-link} to find a Hamiltonian cycle $H_{l}^{(j)}$ in $Y_{l}^{(j)}$ through $f_j$ and an edge $f_l^{(j)}$ in $A_{u_{l}}$. Set $F_l^{(j)}=F_{l-1}^{(j)}\cup \{f_l^{(j)}\}$.

Hence, by the above claim, the number of Hamiltonian cycles in $G$ is at least 
$\sum_{j=1}^k |\mathcal{C}_j| \geq \sum_{j=1}^k (k+1-j)=k(k+1)/2> c_2^2 n^2/2=c n^2.$
\end{proof}

\section{Restricting degree 4 vertices}\label{sec:restrict_degree}
In this section, we prove Theorem~\ref{main2}. We first show several lemmas.

\begin{lem}\label{pst:2-4cycles}
Let $G$ be a $4$-connected planar triangulation.
Let $S$ be an independent set in $G$ saturating no $4$-cycle or $5$-cycle in $G$. Let $u,u'\in S$ be distinct and let $D_u$ and $D_{u'}$ be $4$-cycles containing $u$ and $u'$, respectively. Then $|V(D_u)\cap V(D_{u'})|\leq 2$, $V(D_u)\cap V(D_{u'})\cap S=\emptyset$, and if $V(D_u)\cap V(D_{u'})$ consists of two vertices, say $a$ and $b$,  then $a b\in E(D_u)\cap E(D_{u'})$.
\end{lem}
\begin{proof}
 Note that $u'\notin V(D_u)$ since $\{u,u'\}$ saturates no $4$-cycle in $G$. Similarly, $u\notin V(D_{u'})$.   So $V(D_u)\cap V(D_{u'})\cap S=\emptyset$. Moreover,  $|V(D_u)\cap V(D_{u'})|\le 2$ since otherwise $u,u'$ are contained in a $4$-cycle in $G[D_u+u']$, contradicting the assumption that $S$ saturates no $4$-cycle in $G$.

Now suppose $V(D_u)\cap V(D_{u'})=\{a,b\}$ with $a\ne b$. If $ab\in E(D_u)\backslash E(D_{u'})$ then $G[D_u+u']$ has a $5$-cycle containing $u$ and $u'$, contradicting the assumption that $S$ saturates no $5$-cycle in $G$. So, $ab\notin E(D_u)\backslash E(D_{u'})$. Similarly, $ab\notin E(D_{u'})\backslash E(D_{u})$. If $ab\notin E(D_u)\cup E(D_{u'})$, then  $u,u'$ are contained in a $4$-cycle in $G$, a contradiction. Thus, $ab\in E(D_u)\cap E(D_{u'})$.
\end{proof}

Recall that for a cycle $D$ in $G$, $\overline{D}$ is the subgraph of $G$ consisting of all vertices and edges of $G$ contained in the closed disc bounded by $D$.  
In the proof of Theorem~\ref{main2}, we will need to consider
the subgraphs of a planar triangulation that lie between two separating $4$-cycles and use the following result on Hamiltonian paths in those subgraphs. 

\begin{lem}\label{nest4cycles}
Let $G$ be a $4$-connected planar triangulation in which the distance between any two vertices of degree $4$ is at least three.
Let $S$ be an independent set in $G$ such that $S$ saturates no $4$-cycle or $5$-cycle in $G$. Let $u,u'\in S$ be distinct, and $D_u, D_{u'}$ be separating $4$-cycles in $G$ containing $u$ and $u'$, respectively. Suppose $\overline{D_{u'}}\subseteq \overline{D_u}$, and $D_{u'}$ is a maximal separating $4$-cycle containing $u'$ in $G$, i.e., $\overline{D_{u'}}$ is not contained in $\overline{D}$ for any other separating $4$-cycle $D\neq D_{u'}$ with $u'\in V(D)$. Let $H$ denote the graph obtained from $\overline{D_u}$ by contracting  $\overline{D_{u'}}-D_{u'}$ to a new vertex $z$ so that $H$ is a near triangulation with outer cycle $D_u$.
Then one of the following holds:
\begin{itemize}
\item [\textup{(i)}] For any distinct $a,b\in V(D_u)$, $H-(V(D_u)\backslash \{a,b\})$ has at least two Hamiltonian paths between $a$ and $b$. 
\item [\textup{(ii)}] There exist distinct $a,b\in V(D_u)$ such that  $H-(V(D_u)\backslash \{a,b\})$ has a unique Hamiltonian path, say $P$,  between $a$ and $b$; but for any distinct $c,d\in  V(D_u)$ with $\{c,d\}\ne \{a,b\}$, $H-(V(D_u)\backslash \{c,d\})$ has at least two Hamiltonian paths between $c$ and $d$ and avoiding an edge of $P$ incident with $z$.
\end{itemize}
\end{lem}
\begin{proof}
Let $D_u=uvwxu$ and $D_{u'}=u'v'w'x'u'$. Without loss of generality, assume that $u, v, w, x$ occur on $D_u$ in clockwise order, and $u', v', w',  x'$ occur on $D_{u'}$ in clockwise order.

By Lemma~\ref{pst:2-4cycles}, we have  $|V(D_u)\cap V(D_{u'})| \leq 1$ or $|E(D_u)\cap E(D_{u'})|=1$. Thus, $|V(H)|\geq 7$, and for any distinct $a,b \in V(D_u)$ with $ab\notin E(D_u)$,  $H-(V(D_u)\backslash \{a,b\})$ is not a path. So by Lemma~\ref{uw-path}, we have
\begin{claim}\label{cl:notedge_2Hamilton}
For any distinct $a,b\in V(D_u)$ with $ab\notin E(D_u)$, $H-(V(D_u)\backslash \{a,b\})$ has at least two Hamiltonian paths between $a$ and $b$.
\end{claim}

\begin{claim}\label{cl:intersection_notempty}
We may assume $V(D_u)\cap V(D_{u'})\neq \emptyset$.  
\end{claim}
\begin{proof}
For, suppose $V(D_u)\cap V(D_{u'})= \emptyset$. Then $z$ is not incident with the infinite face of $H-D_u$. So for any distinct  $a,b\in V(D_u)$ with $ab\in E(D_u)$, $H-(V(D_u)\backslash \{a,b\})$ cannot be an outer planar graph. Thus, by 
 Lemma~\ref{uv-path}, $H-(V(D_u)\backslash \{a,b\})$ has at least two Hamiltonian paths between $a$ and $b$. So (i) holds by Claim~\ref{cl:notedge_2Hamilton}. 
\end{proof}
\begin{claim}\label{cl:intersection_two}
   We may further assume that $|V(D_u)\cap V(D_{u'})|=2$. 
\end{claim}
\begin{proof}
For, suppose $V(D_u)\cap V(D_{u'})$ consists of exactly one vertex, say $y$. Then $y\in \{v,w,x\}\cap \{v',w',x'\}$. We show that (i) holds. By Claim~\ref{cl:notedge_2Hamilton}, it suffices to consider distinct $a,b\in V(D_u)$ with $ab\in E(D_u)$. By Lemma~\ref{uv-path}, it suffices to show that $H-(V(D_u)\backslash \{a,b\})$ is not outer planar. 

Let $a,b\in V(D_u)$ with $ab\in E(D_u)$. If $y\in \{a,b\}$, then $z$ is not incident with the infinite face of $H-(V(D_u)\backslash \{a,b\})$; so $H-(V(D_u)\backslash \{a,b\})$ is not outer planar. Hence we may assume that $y\notin \{a,b\}$. Let $D_u=yy_1y_2y_3y$ and assume that $y,y_1,y_2,y_3$ occur on $D_u$ in clockwise order. Then $\{a,b\}=\{y_1,y_2\}$ or $\{a,b\}=\{y_2,y_3\}$.

First, assume that $y\in \{v',x'\}$. We consider $y=v'$, as the other case $y=x'$ is symmetric. If $H-\{y,y_3\}$ is outer planar, then $u'$ is adjacent to the vertices $y,y_3$ in $V(D_u)$. Since $yy_3\in E(D_u)$, $G[D_u+u']$ has a 5-cycle containing $u$ and $u'$, contradicting the assumption that $S$ saturates no 5-cycle. Hence, $H-(V(D_u)\backslash \{y_1,y_2\})=H-\{y,y_3\}$ is not outer planar. It remains to consider $H-(V(D_u)\backslash \{y_2,y_3\})=H-\{y,y_1\}$. Suppose $H-\{y,y_1\}$ is outer planar. Then $w'$ and $x'$ are incident with the infinite face of $H-\{y,y_1\}$ and $w'y_1\in E(G)$. We claim that $x'y_1\in E(G)$; otherwise $x'y\in E(G)$, implying that $x'w'yx'$ or $x'u'yx'$ is a separating triangle in $G$, a contradiction.
But then $D=u'yy_1x'u'$ is a separating 4-cycle in $G$ containing $u'$, and $\overline{D}$ properly contains $\overline {D_{u'}}$, contradicting the maximality of $D_{u'}$.

Suppose $y=w'$. For $\{a,b\}=\{y_1,y_2\}$ or $\{a,b\}=\{y_2,y_3\}$, if $H-(V(D_u)\backslash \{a,b\})$ is not outer planar, then $u'y_3\in E(G)$ or $u'y_1\in E(G)$. So $D=u'y_3w'x'u'$ or $D=u'y_1w'v'u'$ is a separating $4$-cycle in $G$ such that $\overline{D}$ properly contains $\overline{D_{u'}}$. Thus, $H-(V(D_u)\backslash \{a,b\})$ cannot be outer planar for $\{a,b\}=\{y_1,y_2\}$ or $\{a,b\}=\{y_2,y_3\}$.
\end{proof}

By Claim~\ref{cl:intersection_two}, $|E(D_u)\cap E(D_{u'})|=1$; so $V(D_u)\cap V(D_{u'})=\{v,w\}$ or  $V(D_u)\cap V(D_{u'}) =\{w,x\}$. By the symmetry among the edges in $D_u$ and between the two orientations of $D_u$, we may further assume  $V(D_u)\cap V(D_{u'})=\{v,w\}$.
\begin{claim}\label{cl:2Hamilton_exceptone}
For $\{a,b\}\subseteq V(D_u)$ with $ab\in E(D_u)$, if $\{a,b\}\neq \{u,x\}$, then $H-(V(D_u)\backslash \{a,b\})$ has at least two Hamiltonian paths between $a$ and $b$.
\end{claim}
\begin{proof}
For $\{a,b\}=\{v,w\}$, since $z$ is not incident with the infinite face of $H-\{x,u\}$, $H-(V(D_u)\backslash \{a,b\})=H-\{x,u\}$ is not outer planar and has at least two Hamiltonian paths between $a$ and $b$ by Lemma~\ref{uv-path}.

For $\{a,b\}=\{u,v\}$ or $\{a,b\}=\{w,x\}$, $H-(V(D_u)\backslash \{a,b\})$ cannot be outer planar. Otherwise, one can check that $\{u,u'\}$ is contained in a 4-cycle or 5-cycle in $G$. Hence, by Lemma~\ref{uv-path}, there exist at least two Hamiltonian paths between $a$ and $b$ in $H-(V(D_u)\backslash \{a,b\})$ when $\{a,b\}=\{u,v\}$ or $\{a,b\}=\{w,x\}$.
\end{proof}

By Claim~\ref{cl:2Hamilton_exceptone}, we may assume that $H-\{v,w\}$ has a unique Hamiltonian path $P$ between $u$ and $x$, as otherwise (i) holds. It follows from Lemma~\ref{uv-path} that $H-\{v,w\}$ is an outer planar near triangulation. Let $y'$ denote the vertex in $V(D_{u'})\backslash \{u',v,w\}$; so $y'=x'$ or $y'=v'$. Note $V(uPz)\subseteq N_H(v)$, $V(xPz)\subseteq N_H(w)$, and $P$ contains $u'zy'$. Let $r$ denote the unique vertex in $N_H(u)\cap N_H(x)$. Observe that $\{v,w\}=\{v',w'\}$ or $\{v,w\}=\{w',x'\}$. Recall the definition of diamond-$4$-cycle in Figure~\ref{diamond cycle}.

\begin{claim}\label{cl:Du'y_diamond_4}
There exists a vertex $y\in V(P) \backslash \{u,x,z\}$ such that $yu'\in E(P)$ and $D':=G[D_{u'}+ y]$ is a diamond-$4$-cycle with $u'$ and $y$ as crucial vertices. Moreover, $r\notin \{y,y'\}$.
\end{claim}
\begin{proof}
First, suppose $\{v,w\}=\{v',w'\}$, i.e. $v=v'$ and $w=w'$. Then $y'=x'$. Hence, there exists a vertex $y$ in $V(uPu') \backslash \{u,u'\}$ such that $yu'\in E(P)$ and $yv\in E(G)$. 
If $yx'\notin E(G)$, then $u'$ has a neighbor $z'$ in $V(xPx')$ since $H-\{v,w\}$ is an outer planar near triangulation; now $u'v'w'z'u'$ is a separating 4-cycle in $G$ containing $u'$ (as $z'w=z'w'\in E(G)$), contradicting the maximality of $D_{u'}$. Therefore, $yx'\in E(G)$ and $G[D_{u'}+ y]$ is a diamond-$4$-cycle with crucial vertices $u'$ and $y$.
Moreover, $r\notin \{y,x'\}=\{y,y'\}$; otherwise, $uvu'yu$ (when $r=y$) or $uvu'x'u$ (when $r=y')$ is a 4-cycle saturated by $S$, a contradiction.

Now assume that $\{v,w\}=\{w',x'\}$, i.e., $v=w'$ and $w=x'$. Then $y'=v'$.
Observe that $u'x\notin  E(G)$, otherwise $uvwu'xu$ is a $5$-cycle in $G$ saturated by $S$, a contradiction. Hence, there exists $y\in V(u'Px)\backslash \{u',x\}$ such that $yu'\in E(P)$ and $yw\in E(G)$. Now $yv'\in E(G)$ by the maximality of $D_{u'}$. Therefore, $G[D_{u'}+ y]$ is a diamond-$4$-cycle in $G$ in which $u',y$ are crucial vertices. If $r=y'=v'$ then $u v w u' y' u$ is a $5$-cycle in $G$ containing $\{u,u'\}$, and if $r=y$ then $uy u' w x u$ is a $5$-cycle in $G$ containing $\{u,u'\}$. This contradicts the assumption that $S$ saturates no $5$-cycle in $G$, completing the proof of Claim \ref{cl:Du'y_diamond_4}.
\end{proof}

 We need another claim, in order to show that for any $\{c,d\}\neq \{u,x\}$, $H-(V(D_u)\backslash \{c,d\})$ has at least two Hamiltonian paths between $c$ and $d$ and not containing $u'zy'$.
Let $H':=H-(V(D_u)\cap V(D_{u'})\cup \{z\})=H-\{v,w,z\}$. Then $H'$ is an outer planar near triangulation and $H'\subseteq G$. 

\begin{claim}
  $r\in V(H')\backslash\{y,u',y'\}$ and $H'-\{u,x,u',y'\}$ has a Hamiltonian path $P_1$ between $r$ and $y$.
\end{claim}
\begin{proof}
Since $r\in N_G(u)$ and $S$ is independent, $r\neq u'$.
By Claim~\ref{cl:Du'y_diamond_4}, $r\notin \{y,y'\}$ and $y\notin \{u,x,y'\}$. Thus, $r\notin \{y,y',u'\}$. 

Let $C$ denote the outer cycle of $H'$. Then $ux, u'y'\in E(C)$. 
We may assume $y'=x'$; the other case is similar.

Since $G$ contains no separating triangle, each edge in $H'-E(C)$ is incident with both $uPu'$ and $xPy'$.
Since $V(uPu')\subseteq N_G(v)$ and $V(xPy')\subseteq N_G(w)$, every degree $4$ vertex of $G$ in $V(H')\backslash \{u,x,u',y'\}$ has degree $3$ in $H'$. Hence, by assumption of the lemma, the distance between any two degree $3$ vertices of $H'$, contained in $V(H')\backslash\{u,x,u',y'\}$, is at least three in $H'$. Applying Lemma~\ref{lem:r-yHamiltonian} to $H'$, we see that $H'-\{u,x,u',y'\}$ has a Hamiltonian path $P_1$ between $r$ and $y$.
 \end{proof}

 Let $Q_1:=P_1\cup yu'y'z$ and $Q_2:=P_1\cup yy'u'z$. Then $Q_1$ and $Q_2$ are two distinct Hamiltonian paths between $r$ and $z$ in $H-V(D_u)$, and neither contains $u'zy'$. 
We now show that (ii) holds with $\{a,b\}=\{u,x\}$. Let $c, d\in V(D_u)$ be distinct such that $\{c,d\}\neq \{u,x\}$. Observe that one vertex in $\{c,d\}$ is a neighbor of $r$ and the other is a neighbor of $z$. We may assume $c\in N_H(r)$ and $d\in N_H(z)$. Then $cr\cup Q_1\cup zd, cr \cup Q_2 \cup zd$ are two distinct Hamiltonian paths in $H-(V(D_u)\backslash \{c,d\})$ between $c$ and $d$ and not containing $u'zy'$.
\end{proof}

We also need the following result, which is given implicitly in the proof of Theorem 1.3 in \cite{Liu-Yu2021}.

\begin{lem}[Liu and Yu \cite{Liu-Yu2021}]\label{nest4cycles_1}
Let $G$ be a $4$-connected planar triangulation. Assume that $G$ contains a collection of separating $4$-cycles, say ${\cal D}=\{D_1, D_2, \ldots, D_{t+1}\}$, such that $\overline{D_1}\supseteq \overline {D_2}\supseteq \cdots \supseteq \overline{D_{t+1}}$. For $j\in [t]$, let $G_j$ be the graph obtained from $\overline{D_j}$ by contracting $\overline{D_{j+1}}-D_{j+1}$ to a new vertex, denoted by $z_{j+1}$. Suppose the conclusion of Lemma~\ref{nest4cycles} holds for $G_j$ and $z_{j+1}$ (as $H$ and $z$, respectively, in Lemma~\ref{nest4cycles}). Then $G$ has at least $2^{\sqrt{t}}$ Hamiltonian cycles.

\end{lem}

\begin{proof}[Proof of Theorem~\ref{main2}] Note that, for any two distinct vertices $x,y$ of degree $4$ in $G$, we have $N_G(x)\cap N_G(y)=\emptyset$ as $d_G(x,y)\geq 3$. Hence, the number of vertices of degree $4$ in $G$ is at most $n/5$. Thus, since $|E(G)|=3n-6$ and $\delta(G)\geq 4$, there exist at least $n/5$ vertices of degree $5$ or $6$ in $G$. Then by the Four Color Theorem, there is an independent set $I$ such that  every vertex in $I$ has degree $5$ or $6$ in $G$ and $|I|\geq (n/5)/4 = n/20$. We may assume that
\begin{itemize}
    \item [(1)] $G$ has an independent set $S\subseteq I$ of size $\Omega(n^{3/4})$ such that $S$ saturates no $4$-cycle, or $5$-cycle, or diamond-$6$-cycle in $G$.
\end{itemize}
For, otherwise, by Lemma~\ref{2cases}, there exist distinct $v,x\in V(G)$ such that $|N_G(v)\cap N_G(x)\cap I|\geq c_0 n^{1/4}$ for some constant $c_0>0$. Since any two vertices of degree $4$ in $G$ have distance at least three, $G[N_G[v]\cup N_G[x]]$ contains separating $4$-cycles $C_1, \ldots, C_k$ in $G$, where $k \geq c_0 n^{1/4}-1$, such that $|V(\overline{C_i})|\geq 6$ for each $i\in [k]$, and $\overline{C_i}-C_i$, $\overline{C_j}-C_j$ are disjoint whenever $1\leq i\ne j \leq k$. Let $G^*$ be the graph obtained from $G$ by contracting $\overline{C_i}-C_i$ to a new vertex $v_i$, for $i\in [k]$. 
Then $G^*$ is a $4$-connected planar triangulation and, hence, has a Hamiltonian cycle, say $H$. 

Let $a_i,b_i\in N_{G^*}(v_i)$ such that $a_iv_ib_i\subseteq H$ for $i\in [k]$. 
Since $|V(\overline{C_i})|\geq 6$ and no vertices of degree $4$ in $G$ are adjacent, it follows from Lemma~\ref{lem:2Hamiltonpaths} that $\overline{C_i}-(V(C_i)\backslash \{a_i,b_i\})$ has at least two Hamiltonian paths between $a_i$ and $b_i$. We can form a Hamiltonian cycle in $G$ by taking the union of $H-\{v_i: i\in [k]\}$ and one Hamiltonian path  between $a_i$ and $b_i$ in  $\overline{C_i}-(V(C_i)\backslash \{a_i,b_i\})$ for each $i\in [k]$. Thus, $G$ has at least $2^k \ge 2^{c_0n^{1/4}-1}$ Hamiltonian cycles and we are done. This completes the proof of (1).

\medskip

For each $u\in S$, recall the link $A_u$ defined in Section 2. 
We may assume that
\begin{itemize}
    \item [(2)] there exists $S_1\subseteq S$ such that $|S_1|\geq |S|/2$ and, for each $u\in S_1$, $d_G(u)-|A_u|\geq 2$ and $u$ is contained in a separating $4$-cycle $D$ in $G$ with $|V(\overline{D})|\geq 6$.
\end{itemize}
Suppose we have $S_2\subseteq S$ with $|S_2|\geq|S|/2$ such that $d_G(u)-|A_u|\leq 1$ for all $u\in S_2$. Hence, for any $u\in S_2$, $|A_u|\geq 4 $ if $d_G(u)=5$; and $|A_u|\geq 5$ if $d_G(u)=6$. Let $F$ be any subset of $E(G)$ with $|F|=|S_2|$ and $|F\cap A_u|=1$ for each $u\in S_2$. By Lemma~\ref{4conn}, $G-F$ is $4$-connected; so $G-F$ has a Hamiltonian cycle by Tutte's theorem. Let $\cal{C}$ be a collection of Hamiltonian cycles in $G$ by taking precisely one Hamiltonian cycle in $G-F$ for each choice of $F$. Let $a_1$ and $a_2$ denote the number of vertices in $S_2$ of degree $5$ and $6$ in $G$, respectively. There are at least $4^{a_1}5^{a_2}$ choices of the edge set $F\subseteq E(G)$. Each Hamiltonian cycle of $G$ in ${\cal C}$ is chosen at most $(5-2)^{a_1}(6-2)^{a_2}=3^{a_1}4^{a_2}$ times. Thus $|{\cal C}|\ge (4/3)^{a_1}(5/4)^{a_2}\geq (5/4)^{a_1+a_2}=(5/4)^{|S_2|}\geq (5/4)^{\Omega(n^{3/4})}$.

Hence, we may assume that there exists $S_1\subseteq S$ such that $|S_1|\geq |S|/2$ and $d_G(u)-|A_u|\geq 2$ for all $u\in S_1$.
For each $u\in S_1$, since $d_G(u)-|A_u|\geq 2$, there exist at least two edges $e_1$ and $e_2$ incident with $u$ such that $G-e_i$ is not $4$-connected for $i\in [2]$. Since $u$ has at most one neighbor of degree $4$ in $G$ (by assumption), there exists $i\in [2]$ such that a $3$-cut of $G-e_i$ and $u$ induce a separating $4$-cycle $D_u$ in $G$ with $|V(\overline{D_u})|\geq 6$. This completes the proof of (2).
 
 \medskip
 
 For each $u\in S_1$, we choose a maximal separating $4$-cycle $D_u$ containing $u$. Note $|V(\overline{D_u})|\geq 6$.
 Let ${\cal D}=\{D_u: u\in S_1\}$. Since $S_1$ saturates no $4$-cycle, $D_u\neq D_{u'}$ for any distinct $u,u'\in S_1$ and $|{\cal D}|=|S_1|\geq |S|/2$. By Lemma~\ref{pst:2-4cycles}, for any distinct $D_1,D_2\in{\cal D}$, either $\overline{D_1}-D_1$ and $\overline{D_2}-D_2$ are disjoint, or $\overline{D_1}$ contains $\overline{D_2}$ or vice versa. We may assume that 
\begin{itemize}
    \item [(3)] there exist $D_1, D_2,\dots , D_{t+1}\in {\cal D}$, where $t=\Omega({n^{1/2}})$, such that $\overline{D_1}\supseteq \overline{D_2 }\supseteq \dots \supseteq \overline{D_{t+1}}$.
\end{itemize}
For, otherwise, since $|{\cal D}|= {\color{blue} |S_1|\geq |S|/2=}\Omega({n^{3/4}})$, there exist separating 4-cycles $D_1',\ldots, D_k'\in \cal D$, where $k=\Omega({n^{1/4}})$, such that $|V(\overline{D_i'})|\geq 6$ for $i\in [k]$, and  $\overline{D_i'}-D_i'$, $\overline{D_j'}-D_j'$ are disjoint for $1\leq i \neq j \leq k$. Hence, $G$ has at least $2^k$ Hamiltonian cycles, as shown in the first paragraph in the proof of (1). This completes the proof of (3).

\medskip

For each $j\in [t]$, 
let $G_j$ denote the graph obtained from $\overline{D_j}$ by contracting $\overline{D_{j+1}}-D_{j+1}$ to a new vertex $z_{j+1}$. Note that $G_j$ is a near triangulation with outer cycle $D_j$ and that $G_j$ contains the $4$-cycle $D_{j+1}$. 

By Lemma~\ref{pst:2-4cycles} and the definition of $\cal {D}$, we see that $D_{j+1}$, $D_j$, $G_j$, and $G$ (as $D_{u'}, D_{u}, H, G$, respectively,  in Lemma~\ref{nest4cycles}) for $j\in[t]$, satisfy the conditions in Lemma~\ref{nest4cycles}. Hence, by Lemma~\ref{nest4cycles} and Lemma~\ref{nest4cycles_1}, $G$ has at least $2^{\sqrt{t}}=2^{\Omega(n^{1/4})}$ Hamiltonian cycles. 
\end{proof}


\begin{thebibliography}{10}

\bibitem {AAT2020} A. Alahmadi, R. Aldred, and C. Thomassen, Cycles in 5-connected triangulations, {\it J. Combin. Theory Ser. B} {\bf 140} (2020) 27--44.  

\bibitem{AACST2013a}
A. Alahmadi, R. E. L. Aldred, R. dela Cruz, P. Sol\'e, and C. Thomassen, The maximum number of minimal codewords in long codes, \textit{Discrete Appl. Math.} \textbf{161} (2013) 424--429.

\bibitem{AACST2013b}
A. Alahmadi, R. E. L. Aldred, R. dela Cruz, P. Sol\'e, and C. Thomassen, The maximum number of minimal codewords in an $[n,k]$-code, \textit{Discrete Math.} \textbf{313} (2013) 1569--1574.

\bibitem{AACOST2015}
A. Alahmadi, R. E. L. Aldred, R. dela Cruz, S. Ok, P. Sol\'e, and C. Thomassen, The minimum number of minimal codewords in an $[n,k]$-code and in graphic codes, \textit{Discrete Appl. Math.} \textbf{184} (2015) 32--39.

\bibitem{Altshuler1972}
A. Altshuler, Hamiltonian circuits in some maps on the torus, \textit{Discrete Math.} \textbf{1} (1972) 299--314.

\bibitem {BHT1999} T. B\"{o}hme, J. Harant, and M. Tk{\' a}{\v c}, On certain Hamiltonian cycles in planar graphs, {\it J. Graph Theory} {\bf 32} (1999) 81--96.


\bibitem{BSC2018}
G. Brinkmann, J. Souffriau, and N. Van Cleemput, On the number of Hamiltonian cycles in triangulations with few separating triangles, \textit{J. Graph Theory} \textbf{87} (2018) 164--175.

\bibitem {BC2021}
G. Brinkmann and N. Van Cleemput, $4$-connected polyhedra have at least a linear number of hamiltonian cycles, {\it European J. Combin.} {\bf 97} (2021) 103395.

\bibitem{Brunet-Richter1995}
R. Brunet and R. B. Richter, Hamiltonicity of 5-connected toroidal triangulations, \textit{J. Graph Theory} \textbf{20} (1995) 267--286.



\bibitem{HST1979}
S. L. Hakimi, E. F. Schmeichel, and C. Thomassen, On the number of Hamiltonian cycles in a maximal planar graph, \textit{J. Graph theory} \textbf{3} (1979) 365--370.

\bibitem{Jackson-Yu2002}
B. Jackson and X. Yu, Hamilton cycles in plane triangulations, \textit{J. Graph Theory} \textbf{41} (2002) 138--150.

\bibitem{Liu-Yu2021}
X. Liu and X. Yu, Number of Hamiltonian cycles in planar triangulations, \textit{SIAM J. Discrete Math.} {\bf 35(2)} (2021) 1005--1021.


\bibitem{Lo2020}
O. S. Lo, Hamiltonian cycles in $4$-connected plane triangulations with few $4$-separators, \textit{Discrete Math.} \textbf{343} (2020) 112126.

\bibitem{LQ2021}
O. S. Lo and J. Qian, Hamiltonian cycles in $4$-connected planar and projective planar triangulations with few $4$-separators, {\it arXiv:2104.12481}.

\bibitem{Thomas-Yu1994}
R. Thomas and X. Yu, $4$-connected projective planar graphs are Hamiltonian, \textit{J. Combin. Theory Ser. B} \textbf{62} (1994) 114-132.

\bibitem{Thomas-Yu1997}
R. Thomas and X. Yu, $5$-connected toroidal graphs are Hamiltonian, \textit{J. Combin. Theory, Ser. B} \textbf{69} (1997) 79--96.

\bibitem{TYZ2005}
R. Thomas, X. Yu, and W. Zang,  Hamilton paths in toroidal graphs, \textit{J. Combin. Theory, Ser. B} \textbf{69} (2005) 214--236.

\bibitem{Thomassen1983}
C. Thomassen, A theorem on paths in planar graphs, \textit{J. Graph Theory} \textbf{7} (1983) 169--176.

\bibitem{Tutte1956}
W. T. Tutte, A theorem on planar graphs, \textit{Trans. Amer. Math. Soc.} \textbf{82} (1956) 99--116.

\bibitem{Whitney1931}
H. Whitney, A theorem on graphs, \textit{Ann. Math.} \textbf{32(2)} (1931) 378--390.


\end{thebibliography}
\end{document}